\documentclass[a4paper, oneside]{amsart}
\usepackage[foot]{amsaddr}
\usepackage[T1]{fontenc}
\usepackage{geometry}
\usepackage[utf8]{inputenc}
\usepackage{graphicx}
\usepackage{dsfont}
\usepackage{amsmath,amssymb,amsthm}
\usepackage{mathrsfs}
\usepackage{enumitem}
\usepackage{etoolbox}
\usepackage{subcaption}
\usepackage{multirow}
\usepackage{pdfpages}
\usepackage{booktabs}
\usepackage{mathtools}
\usepackage{tgpagella}
\usepackage{todonotes}
\usepackage[normalem]{ulem}
\usepackage[hidelinks]{hyperref}
\usepackage[most]{tcolorbox}

\newcommand{\myemph}[1]{\uline{#1}}

\newcommand{\bpxcoef}{\mu_{\ell,\delta}}
\title[Low rank tensor approximation of singularly perturbed 
  PDEs in one dimension]{Low rank tensor approximation of singularly perturbed partial
  differential equations in one dimension}
\author{Carlo Marcati$^\star$} 
\email{carlo.marcati@sam.math.ethz.ch}
\author{Maxim Rakhuba$^\dagger$}
\email{mrakhuba@hse.ru}
\author{Johan E.~M.~Ulander$^\star$}
\email{ulanderj@student.ethz.ch}
\address{$^\star$Seminar for
Applied Mathematics, ETH Zürich, 8092 Zürich, Switzerland}
\address{$^\dagger$National Research University Higher School of Economics,
  109028 Moscow, Russia}

\makeatletter
\@namedef{subjclassname@2020}{\textup{2020} Mathematics Subject Classification}
\makeatother
\newcommand{\range}[1]{\{1, \dots, #1\}}
\newcommand{\rangezero}[1]{\{0, \dots, #1\}}
\newcommand{\ud}{u_\delta}
\newcommand{\vd}{v_\delta}
\newcommand{\ed}{e_\delta}
\newcommand{\udqtt}{u_\delta^{\mathrm{qtt}}}
\newcommand{\udqttvec}{u_{\mathrm{vec},\delta}^{\mathrm{qtt}}}
\newcommand{\hp}{\mathrm{hp}}
\newcommand{\udp}{u_{\delta, p}}
\newcommand{\udL}{u_{\delta, L}}
\newcommand{\vdL}{v_{\delta, L}}
\newcommand{\PiLtwo}{\Pi_L^{L^2}}
\newcommand{\Cstab}{C_{\mathrm{stab}}}
\newcommand{\Ndof}{N_{\mathrm{dof}}}
\newcommand{\cO}{\mathcal{O}}
\newcommand{\epstol}{\epsilon_{\mathrm{tol}}}
\newcommand{\TunifL}{\mathcal{T}_L}
\newcommand{\TtildeL}{\mathcal{T}^{\mathrm{int}}_L}
\newcommand{\TunifintL}{\mathcal{T}^{\mathrm{int}}_L}

\newcommand{\tA}{\widetilde{A}}
\newcommand{\tf}{\widetilde{f}}
\newtheorem{theorem}{Theorem}
\newtheorem{lemma}{Lemma}
\newtheorem{corollary}[theorem]{Corollary}
\newtheorem{proposition}{Proposition}
\newtheorem{definition}{Definition}

\newtheorem{example}{Example}

\let\epsilon\varepsilon

\let\phi\varphi

\begin{document}
\begin{abstract}
 We derive rank bounds on the quantized tensor train (QTT) compressed approximation of singularly
 perturbed reaction diffusion partial differential equations (PDEs) in one dimension. Specifically, we show
 that, independently of the scale of the singular perturbation parameter, a
 numerical solution with accuracy $0<\epsilon<1$ can be represented in QTT
 format with a number of parameters that depends only polylogarithmically on
 $\epsilon$. In other words, QTT compressed solutions converge exponentially to
 the exact solution, with respect to a root of the number of parameters.
 We also verify the rank bound estimates numerically, and overcome known stability issues of
 the QTT based solution of PDEs by adapting a preconditioning strategy to obtain stable schemes at all scales.
  We find, therefore, that the QTT based strategy is a
  rapidly converging algorithm for the solution of singularly perturbed PDEs,
  which does not require prior knowledge on the scale of the singular perturbation
  and on the shape of the boundary layers.
\end{abstract}
\maketitle

{\small
\tableofcontents
}

\section{Introduction}
\label{sc:intro_TT}
The solution of singularly perturbed elliptic differential equations constitutes
a challenge for numerical approximation. The solutions to such problems exhibit
\myemph{boundary layers}, whose correct resolution is crucial to the
accurate approximation of the problem. Since those layers can get arbitrarily
small, and the variations in the gradient of the solution can get consequentially
highly concentrated in space, the accurate solution of singularly perturbed problems, by,
e.g., low-order finite element (FE) methods requires a computationally
demanding  number of degrees of freedom. For this reason, more effective methods
have been introduced, such as $hp$-FE methods \cite{Schwab1996}, see also
\cite[Chapter 3]{Schwab1998} and \cite{MelenkPDE}, where some
\textit{a priori} knowledge on the solution is exploited to construct numerical methods requiring a
smaller computational effort. In some instances, especially in high dimension,
the implementation of such methods can still be cumbersome.

In this paper, we discuss and analyze the numerical solution of one-dimensional,
singularly perturbed elliptic equations in tensor-compressed format. Specifically,
we formally approximate the problem using low order, piecewise linear finite elements and
compress the resulting algebraic problem (neither the right-hand side, nor the matrix of this system are formed explicitly in computations) using the quantized tensor train (QTT)
method \cite{MatApproxTT, KhoromskijQuanticsApprox}.  By doing this, we obtain
an approximation accuracy comparable with that of a low-order finite element on
a very fine grid (the so-called \textit{virtual grid}), but we only need a
significantly  smaller number of degrees of freedom to compute and represent the
solution.
We also remark that the QTT based approach to the solution of singularly
perturbed problems does not require \textit{a priori} knowledge on the scale of
the singular perturbation, nor on the explicit form of the layers, as for, e.g.,
enriched spectral methods \cite{Checkroun2020}.

\subsection{Contributions of this paper}
First, we show theoretically
and verify numerically that, for all $0<\epsilon<1$, we can
obtain a QTT approximation of the solution with accuracy $\epsilon>0$,
and that we can represent it with $\cO(\left|\log\epsilon\right|^\kappa)$
parameters, with $\kappa=3$. For a given accuracy, the number of parameters of
the QTT approximation is independent of the singular perturbation parameter.
This is the main theoretical contribution of this
paper, and it is stated in Theorem \ref{QTTapprox}.
Furthermore, this is a theoretical
upper bound: we find, in numerical experiments, that $\kappa$ can be smaller
in practice.

The second contribution of this paper is the adaptation of the preconditioner
introduced in \cite{BachmayrStability} to the singularly perturbed case. The
straightforward application of classic solvers (DMRG \cite{White1992}, AMEn \cite{AMEnLinearSystems}, etc.)
to the QTT formatted tends indeed to have stability issues, which greatly limit
the virtual grid sizes that can be used in practice. In \cite{BachmayrStability},
a BPX preconditioner was developed to overcome this issue; we adapt it to our
case, in order to obtain stable solutions for all values of the perturbation
parameter $0 < \delta < 1$. With this at hand, we are able to reach a virtual grid
size of around $2^{-50}$ and to accurately solve problems with $\delta = 10^{-16}$.
We remark that such an approximation, if represented as a full piecewise linear FE
function, would require approximately $10^{15}$ degrees of freedom, while it is easily
tractable in tensor-compressed, QTT format.

\subsection{Tensor compressed solution of PDEs}
Historically, the first appearance of tensor
decomposition dates back to F. Hitchcook in \cite{Hitchcook}.
In more recent years, a wide range of tensor decompositions have appeared and
have been applied to many fields of science and engineering, see \cite{Kolda2009}.
The tensor train (TT) decomposition, specifically, was introduced in
\cite{OseledetsTTdec} as an easy to construct low-rank decomposition of
high-dimensional matrices and has its roots in matrix product states
representations in physics \cite{Schollwock2011}. 
Shortly later, it was realized
that low-rank tensor representations can handle certain
low-di\-men\-sion\-al partial differential equation (PDEs) that exhibit rough
behavior and are computationally challenging for conventional methods, through
so-called \myemph{quantization}. This refers to the process of reshaping
low-dimensional tensors with high mode sizes into high-dimensional tensor with
small mode size, then applying tensor decomposition. By combining quantization
and the TT representation, one obtains the QTT representation.

The QTT-formatted solution of PDEs proves useful only if the tensors involved
have small QTT-ranks: to theoretically analyse the rank behavior in the problem under
consideration, we approximate the solution with high-order piecewise
polynomials, then $L^2$-project the resulting approximation into the low-order
piecewise linear finite element space. We can then show that the FE function thus
constructed has low exact QTT-ranks (specifically, QTT-ranks that grow only
linearly with respect to the polynomial degree of the high-degree approximation).
The strategy used here is then partially different from the approach in
\cite{Kazeev2018, Marcati2019}, where the high-order piecewise polynomial was
interpolated. $L^2$-projections have the advantage, with respect to
interpolation operators, of being stable with respect to the $\delta$-dependent
norm we compute the error in. It is worth noting that, while the analysis in
multiple dimensions can be significantly more complex than the one presented
here, the strategy to obtain rank bounds can be extended to the
multi-dimensional case.

\subsection{Structure of the paper}
In Section \ref{statement} we introduce the singularly perturbed problem and the
functional setting of the paper.
Section \ref{sc:TT-theory} contains the main theoretical findings of this paper,
i.e., the rank bounds and error analysis for QTT compressed solution of the
singularly perturbed problem. Specifically, we show in Theorem \ref{QTTapprox}
that the number of parameters of the QTT representation of the solution grows
only polylogarithmically with respect to the approximation error.
In Section \ref{numerics} we discuss the numerical stability of the QTT
formatted problem and propose a preconditioning strategy whose
  implementation details are deferred to Appendix \ref{sec:assembly_prec}.
Finally, in Section
\ref{sc:TT-numerics} we present numerical experiments to verify the theoretical
results obtained in Section \ref{sc:TT-theory} and the role of preconditioning.
We conclude and discuss extensions of the present work in Section \ref{sec:conclusion}.
\section{Statement of the problem and notation}
\label{statement}
\subsection{Statement of the problem}
We consider the following
problem on the interval $I = (0,1)$
\begin{equation}
  \begin{aligned}\label{DiffEq}
    &- \delta^{2} u_{\delta}'' + c u_{\delta} = f
    \text{ in } I, \\
    &u_{\delta}(0) = \alpha_0,\, \ud(1) = \alpha_1,
\end{aligned}
\end{equation}
where $0 < \delta < 1$, $\alpha_1, \alpha_0 \in \mathbb{R}$, and where
\begin{equation}
  \label{eq:hypfc}
  f\text{ and }c \text{ are analytic in }\bar{I} = [0,1]\text{ and }
  c(x)\geq c_{\min{}} >0, \, \forall x\in I.
\end{equation}
For the weak formulation of problem
\eqref{DiffEq}, we introduce the Sobolev spaces
\begin{gather*}
  H^{1}(I) \coloneqq \left\{ u \in L^{2}(I) : u' \in L^{2}(I) \right\}, \qquad
  H^{1}_{0}(I) \coloneqq \left\{ u \in H^{1}(I) : u(0) = u(1) = 0 \right\},
\end{gather*}
and 
\begin{gather*}
H^{1}_{D}(I) \coloneqq \left\{ u \in H^{1}(I) : u( 0) = \alpha_0, u(1) = \alpha_1 \right\}.
\end{gather*}
The weak formulation of \eqref{DiffEq} reads then:
find $u_{\delta} \in H^{1}_{D}(I)$ such that
\begin{equation}\label{VarEq}
  a_\delta(\ud, v) \coloneqq \int_{I} \delta^{2} u_{\delta}' v' + \int_{I} c
u_{\delta} v = \int_{I} f v, \quad \forall v \in H^{1}_{0}(I),
\end{equation}
By Lax-Milgram's theorem, for all $\delta >0$, problem \eqref{VarEq} is well defined and has a unique solution.
We introduce, on $H^1(I)$, the $\delta$-dependent norm
\begin{equation}\label{eNorm}
  ||v||_{\delta} := \big( \delta^{2}
||v'||_{L^{2}(I)}^{2} + ||v||_{L^{2}(I)}^{2} \big)^{1/2}.
\end{equation}
We do our analysis in the energy norm just introduced. In the literature, the stronger,
\myemph{balanced norm} is sometimes instead considered, see \cite{Melenk2016}. An
analysis in this norm is out of the scope of the present paper.
\subsection{Notation}
We write $\mathbb{N} = \{1, 2, \dots\}$ for the set of positive natural numbers
and $\mathbb{N}_0  = \mathbb{N}\cup \{0\}$.
For $p\in \mathbb{N}_0$ and $\Omega\subset\mathbb{R}$, let $\mathbb{P}_{p}(\Omega)$ denote the space
of polynomials of degree at most $p$ defined on~$\Omega$.

We use the convention that capitalized letters are used for
multidimensional arrays and non-capitalized letters are used for vectors.

Throughout, if not stated otherwise, we use
the convention that $C>0$ denotes a generic constant independent of the singular
perturbation parameters $\delta$ and of the discretization. $C$ may change value without notice.

With the word tensor we indicate multidimensional arrays: a general
$d$-dimensional tensor is an element $A \in \mathbb{R}^{n_{1} \times \ldots
n_{d}}$, with $n_{i} \in \mathbb{N}$ for $i=1, \ldots ,d$, and requires storage
of order $\cO(n_{1} \cdot \ldots \cdot n_{d})$. 

\section{Low rank QTT approximation}
\label{sc:TT-theory}
In this section, we develop the error analysis and the rank bounds for the
low-rank QTT-formatted approximation of the solutions to \eqref{DiffEq}.
We construct the low-rank representation of solutions $\ud$ to \eqref{DiffEq} by
constructing a piecewise, high-order polynomial approximation to $\ud$, then
reapproximating it in a low-order finite element space, and finally QTT
compressing the resulting vector of coefficients.

We start  by introducing the TT and QTT representations in Section \ref{sec:QTT}.
In Section \ref{sec:polynomial-approximation}, then, we introduce the high-order and the low-order  finite element spaces. Finally, we derive rank bounds and estimate the
approximation error in Sections \ref{sec:rank-bounds} and \ref{sec:error-estimate}, respectively.
\subsection{Quantized Tensor Train}
\label{sec:QTT}
We now formalize the
concepts of \textit{tensor trains} \cite{OseledetsTTdec} and \textit{quantized tensor trains}.
\begin{definition}\label{TT-dec}
  A $d$-dimensional tensor $A \in
\mathbb{R}^{\overbrace{\scriptstyle{n \times \ldots \times
n}}^{d\text{ times}}}$ is said to admit a \myemph{TT-decomposition} if there exist
$\{r_j\}_{j=0}^d\in \mathbb{N}^{d+1}$ such that $r_0=r_{d}= 1$ and that there
exists, for all $j\in \range{d}$,
$V^{j}:\rangezero{n-1} \to \mathbb{R}^{r_{j-1} \times r_{j}}$ such that
\begin{equation}\label{tt-format}
  A(i_{1}, \ldots ,i_{d}) = V^{1}(i_{1}) \ldots
V^{d}(i_{d}),\qquad \forall (i_1, \dots, i_d) \in \rangezero{n-1}^d,
\end{equation}
The tensors $V^{j}$, seen as elements of $\mathbb{R}^{r_{j-1} \times
  n \times r_j}$, are the \myemph{TT-cores} of the decomposition, while
$\{r_j\}_{j=1}^{d-1}$ are the \myemph{TT-ranks} of the decomposition.
\end{definition}
We assume that $r_{2}, \ldots ,r_{d-1}=r \in \mathbb{N}$, for
ease of presentation. The storage required for the representation in
\eqref{tt-format} is of order $\cO(nr^{2}d)$, as each $V^{j}$ can be regarded as a
$3$-tensor in $\mathbb{R}^{r \times n \times r}$ and there are $d$ such
elements. The storage requirement of a $d$-tensor $A \in \mathbb{R}^{n \times
\ldots \times n}$ in the TT-format is polynomial in the dimension $d$, provided
the TT-ranks satisfy $r \leq C d^{k}$ for some $C,k>0$ independent of $d$,
instead of the exponential dependence $\cO(n^{d})$ on $d$ of storage of a tensor
represented as a $d$-dimensional array. 
 
We now introduce the quantized tensor train
(QTT) format. Let $u \in \mathbb{R}^{2^{L}}$, for $L \in \mathbb{N}$, and assume
it is indexed by $i=0, \ldots ,2^{L}-1$. Any index $i \in \{0, \ldots ,2^{L}-1
\}$ admits a binary representation; that is, there exist $i_{j} \in \{0,1\}$,
for $j=0, \ldots ,L-1$, such that
\begin{equation}\label{binRep}
  i = \sum_{j=1}^{L} 2^{L-j} i_{j}.
\end{equation}
We define, using the representation in equation \eqref{binRep},
the $L$-tensor $U \in \mathbb{R}^{2 \times \ldots \times 2}$ such that
\begin{equation}\label{inducedTensor}
  U(i_{1}, \ldots ,i_{L}) \coloneqq u(i), \qquad \forall i =  \sum_{j=1}^{L} 2^{L-j} i_{j} \in \rangezero{2^L-1}
\end{equation}
\begin{definition}\label{QTT-dec}
  A vector $u \in \mathbb{R}^{2^{L}}$ is said to
admit a \myemph{QTT decomposition} \cite{MatApproxTT,KhoromskijQuanticsApprox} if
the corresponding $L$-tensor $U \in \mathbb{R}^{2 \times \ldots \times 2}$
defined in equation \eqref{inducedTensor} admits a TT-decomposition as in
Definition \ref{TT-dec}. The TT-cores of the decomposition of $U$ are called the \myemph{QTT-cores} of
the decomposition of $u$ and the TT-ranks of the decomposition of $U$ are called
the \myemph{QTT-ranks} of the decomposition of $u$. The storage
requirement for such a QTT decomposition is of order $\cO(r^{2}L)$ where $r$ is
the maximal QTT-rank of the decomposition of~$u$.
\end{definition}
The setting that we are interested in is when $u \in
\mathbb{R}^{2^{L}}$ is the coefficient vector of a finite element (FE) function on
a uniform grid. The notion of QTT decompositions can be extended to include
functions $f \colon \mathbb{R} \to \mathbb{R}$:
\begin{definition}
  Let $I \subset \mathbb{R}$ be an open interval, and let
  $\mathcal{T} = \{x_{1} < \ldots < x_{2^{L}} \} \subset \overline{I}$.
  A function $f \colon I \to \mathbb{R}$, well defined at the points of $\mathcal{T}$, is said to admit a \myemph{QTT decomposition} with respect to $\mathcal{T}$ if the vector $v \in \mathbb{R}^{2^{L}}$ with elements
\begin{equation*}
  (v)_{i} = f(x_{i}), \qquad \forall i\in \range{2^L},
\end{equation*}
 admits a QTT-representation. The
QTT-cores of the QTT representation of $v$ are referred to as the \myemph{QTT-cores} of $f$ and the
QTT-ranks of the QTT representation of $v$ are referred to as the \myemph{QTT-ranks} of $f$.
\end{definition}
Certain functions admit exact decompositions in the 
QTT format on equispaced grids $\mathcal{T}$, with bounded ranks. Two fundamental examples of
such functions are given in Example \ref{ex1} and Example \ref{ex2}.
For $L \in
\mathbb{N}$, we denote by
\begin{equation}\label{grid1}
  \TunifL \coloneqq \left\{ x_j =  j \frac{1}{2^{L}+1}, \,
j\in \rangezero{2^{L}+1} \right\}
\end{equation}
the equispaced grid on $\overline{I} = [0,1]$ with $2^{L}+2$ uniformly
spaced grid points.
Also, let 
\begin{equation}\label{grid1int}
  \TunifintL \coloneqq \left\{ x_j =  j \frac{1}{2^{L}+1}, \,
j\in \range{2^{L}} \right\}
\end{equation}
\begin{example}\label{ex1}
  The exponential function $e^{-\alpha x}$, for $\alpha
\in \mathbb{R}$, admits a QTT decomposition with respect to $ \TunifintL$ with QTT-ranks $r=1$, where $L \in
\mathbb{N}$. Any $x \in \TunifintL$, can be written as $x = 
ih$, for some $i \in \range{2^L}$, where $h$ is the step size of $\TunifintL$. Expand $i=\sum_{j=1}^{L} 2^{L-j} i_{j}$, with $i_{j} \in \{ 0,1 \}$,
in its binary representation. Then we obtain
\begin{equation*}
  e^{-\alpha x} =  e^{- \alpha h \big(2^{L-1}
i_{1} + 2^{L-2} i_{2} + \ldots + 2^{1} i_{L-1} + 2^{0} i_{L} \big)} = e^{-\alpha h 2^{L-1} i_{1}} \cdots e^{- \alpha h i_{L}},\qquad \forall x\in \TunifintL,
\end{equation*}
which is a rank-$1$ QTT decomposition.
\end{example}
\begin{example}\label{ex2}
  Any polynomial function $M$ of degree $p$ admits a
  QTT-representation with respect to $\TunifintL$ with QTT-ranks $r=p+1$.
See, e.g., \cite[Theorem
6]{OseledetsConstrApp} for details on the decomposition and \cite{KhoromskijQuanticsApprox} for a proof of this result. A generalization to piecewise polynomials can be found in, e.g., \cite[Lemma 3.7]{SchwabPiecewisePoly} and is
included in the present manuscript as Lemma \ref{PiecewisePolyQTTLemma}.
\end{example}
We refer the reader to \cite{KhoromskijTensorCalc} for a comprehensive
presentation of QTT decomposition.

\subsection{Piecewise polynomial approximations}
\label{sec:polynomial-approximation}
\begin{definition} For any collection $\mathcal{T} = \{ \xi_{0}, \ldots ,\xi_{N} :0 = \xi_{0} < \xi_{1} < \ldots < \xi_{N} = 1 \} \subset \bar{I}$ of ordered
points of $\bar{I}$ and for a degree $p \in \mathbb{N}_{0}$,  we define
\begin{equation*}
  V(I, p, \mathcal{T}) := \{ v \in C(\bar{I})
: v_{|_{[\xi_{i-1},\xi_{i}]}} \in \mathbb{P}_{p}([\xi_{i-1},\xi_{i}]) \text{ for }
i=1, \ldots ,N,\ v(0) = \alpha_0,\, v(1) = \alpha_1 \},
\end{equation*}
and
\begin{equation*}
  V_{0}(I, p, \mathcal{T}) := \{ v \in
C(\bar{I}) : v_{|_{[\xi_{i-1},\xi_{i}]}} \in \mathbb{P}_{p}([\xi_{i-1},\xi_{i}]) 
\text{ for } i=1, \ldots ,N,\ v(0) = v(1)=0 \}.
\end{equation*}
\end{definition}
\subsubsection{High order finite element approximation}
\label{sec:high-order}
We now introduce high-order finite element approximation result that will be essential for
our main result. For all $\kappa >0 $, we denote by $\mathcal{T}^{\hp}_{\kappa} = \{ \xi_{0}, \xi_{1},
\xi_{2}, \xi_{3} \}$ the collection of points defined by
\begin{equation}\label{hpGrid}
  \xi_{0} = 0,\quad \xi_{1} =  \min \{ 0.25, \kappa
\},\quad \xi_{2} = 1 - \min \{ 0.25,\kappa \},\quad \xi_{3} = 1
\end{equation}
\begin{proposition}[{\cite[Proposition 2.2.5]{MelenkPDE}}]\label{hpFEMTheorem}
  Let $\ud$ be the solution to \eqref{DiffEq}. There exist $C, b, \lambda_0>0$
  independent of $\delta$
  such 
  that for every $\lambda \in (0, \lambda_{0})$ and for all $p\in \mathbb{N}_0$,
there exists $\udp \in V(I, p, \mathcal{T}_{\lambda p  \delta}^{\hp})$ such that
\begin{equation*}
  \|\ud - \udp\|_{\delta} \leq C e^{-b p}.
\end{equation*}
\end{proposition}
\subsubsection{$\mathbb{P}_1$ finite element methods}
\label{sec:low-order}
We introduce here, and will use in QTT-compressed computations, a finite element (FE) method, with
piecewise linear basis functions.
We introduce the finite dimensional low-order FE spaces
\begin{equation*}
  V^{L} \coloneqq V(I, 1, \TunifL) \quad \text{and}\quad
  V^{L}_{0} \coloneqq V_0(I, 1, \TunifL).
\end{equation*}
The discretized version of the weak formulation in equation
\eqref{VarEq} reads: find $u_{L} \in V^{L}$ such that
\begin{equation}\label{FEMeq}
  \int_{I} \delta^2 u_{L}' v' + \int_{I} c u_{L} v =
\int_I f v,\qquad \forall v \in V^{L}_{0}.
\end{equation}
The problem in equation \eqref{FEMeq} can be written algebraically as
\begin{equation}\label{sysM}
  A_{L} w_L = f_{L}
\end{equation}
where
\begin{equation*}
  A_{L} \coloneqq \delta^{2} S_{L} + R_{L} \in \mathbb{R}^{2^{L}
    \times 2^{L}},\qquad w_L, f_L \in \mathbb{R}^{2^{L}}.
\end{equation*}
and where $S_{L}\in \mathbb{R}^{2^{L} \times 2^{L}}$ is the stiffness matrix
and $R_L\in \mathbb{R}^{2^{L} \times 2^{L}}$ is the matrix such that
\begin{equation*}
  (R_L)_{ij} = \int_I c \phi_i \phi_j, \qquad \forall i, j\in \range{2^L},
\end{equation*}
with $\{\phi_i\}_{i=1}^{2^L}$ being the Lagrange basis associated with $\TunifintL$.
$A_{L}$ is called the system matrix and $f_{L}$ is called the
load vector of the system in \eqref{FEMeq}. See, e.g., \cite{Ern2004, Brenner2008} for more
details on finite element methods.

Let us now introduce the $\mathbb{P}_{1}$-FE Galerkin projection
$\Pi_{L} \colon H^{1}(I) \to V^{L}$ such that, for all $v\in H^1(I)$,
\begin{equation}\label{GalerkinOrtho}
  a_{\delta}(\Pi_{L}v - v,v_{L}) = 0,\qquad \forall v_{L} \in V^{L}.
\end{equation}
We also introduce the $L^2(I)$ projection $\PiLtwo : L^2(I)\to V^L$, such that,
for all $v\in L^2(I)$,
\begin{equation}
  \label{eq:PiLtwo}
  \int_I\left(\PiLtwo v - v \right) v_L = 0, \qquad \forall v_L\in V^L.
\end{equation}
We remark that the $L^2(I)$ projection is stable with respect to the $H^1(I)$
and $L^2(I)$ norms, hence, there exists a positive constant $\Cstab$ such that, for
all $0<\delta<1$ and for all $L\in\mathbb{N}$,
\begin{equation}
  \label{eq:PiLtwostability}
  \| \PiLtwo v \|_{\delta}\leq \Cstab \|v\|_\delta, \qquad \forall v\in H^1(I).
\end{equation}
\subsection{Rank bounds for QTT approximation}
\label{sec:rank-bounds}
We now wish to establish that there exists a constant $C>0$ independent of $\delta$
such that for all $L\in \mathbb{N}$, $\PiLtwo \udL$ admits a QTT representation with QTT-ranks bounded by
$CL$. First, recall the following lemma on exact
low-rank QTT-representation of piecewise polynomials.
\begin{lemma}[{\cite[Lemma
    3.7]{SchwabPiecewisePoly}}]\label{PiecewisePolyQTTLemma}
  Let $L, M \in
\mathbb{N}$, $p_{1}, \ldots ,p_{M} \in \mathbb{N}_{0}$ and let $x_{0}, \ldots
,x_{M} \in \mathbb{R}$ be such that $0 = x_{0} < \ldots < x_{M} = 2^{L}-1$.
Consider a function $u$ such that $u$ is equal to a polynomial $P_{m}$ of degree
$p_{m}$ in $[x_{m-1},x_{m})$ for $1 \leq m \leq M$ and such that $u(x_{M}) =
P_{M}(x_{M})$. Then the $2^{L}$-component vector $\mathbf{u} = (u_{0}, \ldots
,u_{2^{L}-1})$ with $u_{i} = u(i)$ for $i=0, \ldots ,2^{L}-1$ has a QTT
representation with ranks bounded by $P+M$, where $P= \max \{ p_{1}, \ldots
,p_{M} \} \in \mathbb{N}_{0}$.
\end{lemma}

Next, we need three
auxiliary lemmas.
\begin{lemma}\label{polyLemma}
  For any $P \in \mathbb{P}_{p}(I)$, with $p \in
\mathbb{N}$ and $I \subset \mathbb{R}$, there exist $q_{1,j} \in
\mathbb{P}_{j}(I)$ and $q_{2,j} \in \mathbb{P}_{p-j}(I)$ for $j=0, \ldots ,p$
such that for every $x,y \in I$ with $x+y \in I$
\begin{equation*}
  P(x+y) = \sum_{j=0}^{p} q_{1,j}(x) q_{2,j}(y)
\end{equation*}
\end{lemma}
\begin{proof} We prove by induction over $p \in \mathbb{N}$ that for any
polynomial $P$ of degree $p$ there exist polynomials $q_{k}$ of degree $p-k$,
for $k=0, \ldots ,p$, such that the following equality holds
\begin{equation*}
  P(x+y) = \sum_{k=0}^{p} x^{k} q_{k}(y).
\end{equation*}
The base case $p=1$ is trivial. Suppose that the statement holds
for every polynomial of degree $p-1$. Let $P(x) = \sum_{j=0}^{p} a_{j} x^{j}$ be
a polynomial of degree $p \in \mathbb{N}$. We split the sum as
\begin{equation*}
  P(x+y) = \sum_{k=0}^{p-1} a_{k} (x+y)^{k} + a_{p} (x+y)^{p} =
\sum_{k=0}^{p-1} x^{k} q_{k}(y) + a_{p} (x+y)^{p},
\end{equation*}
where we used that $\sum_{k=0}^{p-1} a_{k} x^{k}$ is a
polynomial of degree $p-1$ and hence we can apply the induction hypothesis to
obtain such $q_{k} \in \mathbb{P}_{p-1-k}(I)$, for $k=0, \ldots ,p-1$. Then, by
the binomial theorem, we have that
\begin{equation*}
  (x+y)^{p} = \sum_{k=0}^{p} \binom{p}{k} x^{k} y^{p-k}.
\end{equation*}
Thus, if we let
\begin{equation*}
  \tilde{q}_{k}(y) \coloneqq q_{k}(y) + a_{p} \binom{p}{k}
y^{p-k} \in \mathbb{P}_{p-k}(I),\ k=0, \ldots ,p-1
\end{equation*}
\begin{equation*}
  \tilde{q}_{p}(y) \coloneqq a_{p} \in \mathbb{P}_{0}(I)
\end{equation*}
then we have
\begin{equation*}
  P(x+y) = \sum_{k=0}^{p} x^{k} \tilde{q}_{k}(y).
\end{equation*}
The assertion follows.
\end{proof}
\begin{lemma}\label{GPLemma1}
  Let $h > 0$ and let $p \in \mathbb{N}$. Suppose
that $\xi_{0} < \ldots . < \xi_{n+1}$ are such that $|\xi_{i+1}-\xi_{i}| > 2h$
for $i=0, \ldots ,n$ and suppose that $q \colon (\xi_{0}, \xi_{n+1}) \to
\mathbb{R}$ is such that $q \in \mathbb{P}_{p}((\xi_{i},\xi_{i+1}))$ for $i=0,
\ldots ,n$. Let $\hat{\phi} \colon (-1,1) \to \mathbb{R}$
and suppose there exists $k \in \mathbb{N}_0$ such that
\begin{equation*}
  \hat{\phi} \in \mathbb{P}_{k}((-1,0)),\ \hat{\phi} \in
\mathbb{P}_{k}((0,1)).
\end{equation*}
Then the function $\Psi:(\xi_0+h, \xi_{n+1}-h)\to \mathbb{R}$ such that
\begin{equation*}
  \Psi(y) = \int_{-1}^{1} \hat{\phi}(x) q(hx+y) dx, \qquad \forall y\in (\xi_0+h, \xi_{n+1}-h)
\end{equation*}
satisfies
\begin{enumerate}
    \item $\Psi \in \mathbb{P}_{p}((\xi_{i}+h,\xi_{i+1}-h))$ for all $i=0,
\ldots ,n$
    \item $\Psi \in \mathbb{P}_{p+k+1}(\xi_{i},\xi_{i}+h)$ for all $i=1, \ldots
,n$
    \item $\Psi \in \mathbb{P}_{p+k+1}((\xi_{i}-h,\xi_{i}))$ for all $i=1,
\ldots ,n$.
\end{enumerate}
\end{lemma}
\begin{proof} Let $I_{i}=(\xi_{i},\xi_{i+1})$ and let $q^{i} \in
\mathbb{P}_{p}(I_{i})$ denote the polynomial such that $q_{|_{I_{i}}} = q^{i}$,
for $i=0, \ldots ,n$. By Lemma \ref{polyLemma}, there exist polynomials
\begin{equation*}
  q^{i}_{1,j} \in \mathbb{P}_{j}(I_{i}),\ q^{i}_{2,j} \in
\mathbb{P}_{p-j}(I_{i}),
\end{equation*}
for $i=0, \ldots ,n$ and for $j=0, \ldots ,p$, such that
\begin{equation*}
  q^{i}(x+y) = \sum_{j=0}^{p} q^{i}_{1,j}(x) q^{i}_{2,j}(y),\
i=0, \ldots ,n,
\end{equation*}
for every $x,y \in I_{i}$ such that $x+y \in I_{i}$. \\ We start
by proving Assertion 1. Fix $i \in \{ 0, \ldots ,n \}$ and let $y \in I_{i}$ be
such that $h x + y \in I_{i}$ for every $x \in (-1,1)$, or equivalently, $y \in
(\xi_{i} +h, \xi_{i+1}-h)$. Then we can write
\begin{equation*}
  \Psi(y) = \sum_{j=0}^{p} \left( \int_{-1}^{1} \hat{\phi}(x)
q^{i}_{1,j}(hx) dx \right) q^{i}_{2,j}(y) = \sum_{j=0}^{p} c_{j} q^{i}_{2,j}(y),
\end{equation*}
where
\begin{equation*}
  c_{j} \coloneqq \int_{-1}^{1} \hat{\phi}(x) q^{i}_{1,j}(hx) dx
\in \mathbb{R},\ j=0, \ldots ,p.
\end{equation*}
Thus, $\Psi \in \mathbb{P}_{p}((\xi_{i}+h,\xi_{i+1}-h))$, which
establishes Assertion 1. \\ We continue to prove Assertion 2 and 3. Fix $i \in
\{ 1, \ldots ,n \}$. Let $x \in (-1,1)$ and let $y \in (\xi_{i}-h,\xi_{i}+h)$.
Then $hx+y \in (\xi_{i}-2h,\xi_{i}+2h)$. As $(\xi_{i}-2h,\xi_{i}+2h) \subset
(\xi_{i-1},\xi_{i+1})$, we have that
\begin{equation*}
  q(hx + y) = \begin{cases} q^{(i-1)}(hx+y), & \text{for } hx+y
\leq \xi_{i} \\ q^{(i)}(hx+y), & \text{for } hx+y>\xi_{i}. \end{cases}
\end{equation*}
Observe that for $x \in (-1,1)$ and $y \in
(\xi_{i}-h,\xi_{i}+h)$
\begin{equation*}
  hx + y < \xi_{i} \iff x \in
\left(-1,\frac{\xi_{i}-y}{h}\right)
\end{equation*}
\begin{equation*}
  hx + y > \xi_{i} \iff x \in \left(\frac{\xi_{i}-y}{h},1
\right).
\end{equation*}
If we denote by
\begin{equation}\label{int-c1}
  c_{1}(y) \coloneqq \sum_{j=0}^{p} \left(
\int_{-1}^{\frac{\xi_{i}-y}{h}} \hat{\phi}(x) q^{i-1}_{1,j}(hx) dx \right)
q^{i-1}_{2,j}(y)
\end{equation}
\begin{equation*}
  c_{2}(y) \coloneqq \sum_{j=0}^{p} \left(
\int_{\frac{\xi_{i}-y}{h}}^{1} \hat{\phi}(x) q^{i}_{1,j}(hx) dx \right)
q^{i}_{2,j}(y),
\end{equation*}
then we can write
\begin{equation} \Psi(y) = \int_{-1}^{\frac{\xi_{i}-y}{h}} \hat{\phi}(x)
q^{i-1}(hx+y) dx + \int_{\frac{\xi_{i}-y}{h}}^{1} \hat{\phi}(x) q^{i}(hx +y) dx
= c_{1}(y) + c_{2}(y),
\end{equation}
Consider the case $y \in (\xi_{i} - h,\xi_{i})$. We prove that
$c_{1} \in \mathbb{P}_{p+k+1}((\xi_{i}-h,\xi_{i}))$. For all $y \in
(\xi_{i}-h,\xi_{i})$, we have that $\frac{\xi_{i}-y}{h} >0$ and so we split the
integral in equation \eqref{int-c1} into two parts
\begin{equation*}
  c_{1}(y) = \sum_{j=0}^{p} \left( \int_{-1}^{0} \hat{\phi}(x)
q^{i-1}_{1,j}(hx) dx + \int_{0}^{\frac{\xi_{i}-y}{h}} \hat{\phi}(x)
q^{i-1}_{1,j}(hx) dx \right) q^{i-1}_{2,j}(y).
\end{equation*}
For each $j=0, \ldots ,p$, the first integral is independent of
$y$
\begin{equation} \int_{-1}^{0} \hat{\phi}(x) q^{i-1}_{1,j}(hx) dx \in \mathbb{R}
\end{equation}
and the second integral is a polynomial of degree $j+k+1$ in $y$:
\begin{equation*}
  \int_{0}^{\frac{\xi_{i}-\cdot}{h}} \hat{\phi}(x)
q^{i-1}_{1,j}(hx) dx \in \mathbb{P}_{j+k+1} \left( (\xi_{i}-h,\xi_{i}) \right).
\end{equation*}
The latter follows from the fact that the integrand is a
polynomial of degree $j+k$ on $\left(0,\frac{\xi_{i}-y}{h} \right)$, as we have
that
\begin{equation*}
  \hat{\phi} \in \mathbb{P}_{k}\left(
\left(0,\frac{\xi_{i}-y}{h} \right) \right)
\end{equation*}
\begin{equation*}
  q^{i-1}_{1,j}(h \cdot) \in \mathbb{P}_{j}\left(
\left(0,\frac{\xi_{i}-y}{h} \right) \right).
\end{equation*}
We conclude that $c_{1} \in
\mathbb{P}_{p+k+1}((\xi_{i}-h,\xi_{i}))$. Similarly, $c_{2} \in
\mathbb{P}_{p+k+1}((\xi_{i}-h,\xi_{i}))$. It follows that $\Psi \in
\mathbb{P}_{p+k+1}((\xi_{i}-h,\xi_{i}))$, the desired property. The case $y \in
(\xi_{i},\xi_{i}+h)$ is proved analogously.
\end{proof}
\begin{lemma}\label{GPLemma2}
  Let $0<\delta<1$, and let $\ud$ be solution to \eqref{DiffEq} under the
  hypotheses \eqref{eq:hypfc}. Let $p\in \mathbb{N}$ and let $\udp$ be the
 approximation of $\ud$ as given in Proposition \ref{hpFEMTheorem}. Let $\{
\phi_{i} \}_{i=1}^{2^{L}}$ be the Lagrange basis associated to $V_0^L$. Then the
vector $v \in \mathbb{R}^{2^{L}}$ such that
\begin{equation*}
  v_{i} = \int_{I} \udp \phi_i, \qquad
i\in \range{2^L},
\end{equation*}
admits a QTT decomposition with QTT-ranks bounded by $p+9$.
\end{lemma}
\begin{proof} Fix $i \in \{ 1, \ldots ,2^{L} \}$, let $h \coloneqq
\frac{1}{2^{L}+1}$ and let $x_{i} =  ih$. Define $\hat{\phi} \colon (-1,1) \to
(0,1)$ by
\begin{equation*}
  \hat{\phi}(x) \coloneqq \phi_{i}(x_{i} + x h).
\end{equation*}
Then $\hat{\phi} \in \mathbb{P}_{1}((-1,0))$ and $\hat{\phi} \in
\mathbb{P}_{1}((0,1))$. Then, by a change of variables,
\begin{equation*}
  v_{i} =
h \int_{-1}^{1} \hat{\phi}(x)
u_{\delta,p}(x h + x_{i}) dx.
\end{equation*}
Moreover, let
\begin{equation*}
  \Psi(y) \coloneqq  h
\int_{-1}^{1} \hat{\phi}(x) \udp(x h+y)dx.
\end{equation*}
Observe that $\Psi(x_{i}) = v_{i}$. Let
\begin{equation*}
  \xi_{0} = -1,\quad \xi_{1} =  \min(0.25,\lambda p \delta),\quad
\xi_{2} = 1 - \min(0.25,\lambda p \delta),\quad \xi_{3} = 1.
\end{equation*}
denote the $hp$-grid as defined in equation \eqref{hpGrid}. We
consider the following two cases:
\begin{enumerate}
\item $\xi_{1}-\xi_{0} = \xi_{3}-\xi_{2} \leq 2h$
\item $\xi_{1}-\xi_{0} = \xi_{3}-\xi_{2} > 2h$. 
\end{enumerate}
\noindent\textbf{Case 1: $\xi_{1}-\xi_{0}=\xi_{3}-\xi_{2} \leq 2h$.}
As $\udp \in
\mathbb{P}_{p}((\xi_{1},\xi_{2}))$, Lemma \ref{GPLemma1} implies that $\Psi \in
\mathbb{P}_{p}((\xi_{1}+h,\xi_{2}-h))$. Extend $\Psi$ to $\tilde{\Psi} \in
\mathbb{P}_{p}((\xi_{0},\xi_{3}))$; that is, such that $\tilde{\Psi}(x) =
\Psi(x)$ for every $x \in (\xi_{1}+h,\xi_{2}-h)$. Let now $\tilde{v} \in
\mathbb{R}^{2^{L}}$ be defined by
\begin{equation*}
  \tilde{v}_{i} \coloneqq \tilde{\Psi}(x_{i}),\ i=1, \ldots
,2^{L}.
\end{equation*}
As $\tilde{\Psi} \in \mathbb{P}_{p}((\xi_{0},\xi_{3}))$, Example
\ref{ex2} implies that $\tilde{v}$ has QTT-ranks bounded by $p+1$. Moreover, for
$i \in \{ 1, \ldots ,2^{L} \}$ such that $x_{i} \in [\xi_{1}+h,\xi_{2}-h]$ we
have that
\begin{equation*}
  \tilde{v}_{i} = \tilde{\Psi}(x_{i}) = \Psi(x_{i}) = v_{i}.
\end{equation*}
The entries of $\tilde{v}$ with $\tilde{v}_{i} \neq v_{i}$ can
be modified by addition or subtraction with rank-1 QTT-vectors in order to be
equal to the corresponding element of $v$. As $\xi_{1}-\xi_{0}=\xi_{3}-\xi_{2}
\leq 2h$, the number of $x_{i}$ with $x_{i} \not\in [\xi_{1}+h,\xi_{2}-h]$ is at
most 4. Thus, we have constructed $v \in \mathbb{R}^{2^{L}}$ as a QTT-vector
with QTT-ranks bounded by $p+5$.
\smallskip\\\noindent
\textbf{Case 2: $\xi_{1}-\xi_{0}=\xi_{3}-\xi_{2} > 2h$.} In this case, we can apply Lemma
\ref{GPLemma1} to each of the subintervals $(\xi_{0},\xi_{1}),(\xi_{1},\xi_{2})$
and $(\xi_{2},\xi_{3})$ (as $\xi_{2}-\xi_{1}>2h$, and
$\xi_{1}-\xi_{0}=\xi_{3}-\xi_{2} > 2h$) to obtain a piecewise polynomial of
degree $p+2$. The fact that $\xi_{2}-\xi_{1} > 2h$ follows from
\begin{equation*}
  \xi_{2} - \xi_{1} = 2 - 2 \min(0.25,\lambda p \delta) \geq \frac{3}{2} > 2h.
\end{equation*}
 By
applying Lemma \ref{GPLemma1} multiple times, we obtain that $\Psi$ is a
polynomial of degree $p$ in the subintervals
\begin{equation*}
  (\xi_{0}+h,\xi_{1}-h),\ (\xi_{1}+h,\xi_{2}-h),\
(\xi_{2}+h,\xi_{3}-h)
\end{equation*}
and a polynomial of degree $p+2$ in the subintervals
\begin{equation*}
  (\xi_{1}-h,\xi_{1}),\ (\xi_{1},\xi_{1}+h),\
(\xi_{2}-h,\xi_{2}),\ (\xi_{2},\xi_{2}+h).
\end{equation*}
Thus, by Lemma \ref{PiecewisePolyQTTLemma}, the vector in
$\mathbb{R}^{2^{L}}$ with elements $\Psi(x_{i})$ has QTT-ranks bounded by
$p+9$. As $v_{i} = \Psi(x_{i})$ this concludes the proof.
\end{proof}
We are now in a position to prove the bound on the QTT-ranks of
$\PiLtwo \udL$.
\begin{proposition}\label{PropQTTranks}
  There exists a constant $C>0$ such that,
for all $0<\delta<1$, for all $L\in \mathbb{N}$, and for all $p\in \mathbb{N}$,
denoting $\ud$ the solution to \eqref{DiffEq}, under the hypotheses
\eqref{eq:hypfc} and with $\alpha_0 = \alpha_1 = 0$, and denoting $\udp$ the
  approximation of $\ud$ as given in Proposition \ref{hpFEMTheorem},
then $\PiLtwo \udp$ admits a QTT decomposition with respect to $\TunifintL$ with QTT-ranks bounded by $Cp$.
\end{proposition}
\begin{proof}
  Let $\{\phi_i\}_{i=1}^{2^L}$ be the Lagrange basis of $V_0(I,1,
  \TunifL)$ and let $v \in \mathbb{R}^{2^{L}}$ be the vector such that
\begin{equation*}
  v_{i} = \int_I\udL \phi_{i},\qquad \forall i\in \range{2^L}.
\end{equation*}
Furthermore, let $M_L$ be the mass matrix associated to the basis $\{\phi_i\}_{i=1}^{2^L}$.
By Lemma \ref{GPLemma2}, $v$ has QTT-ranks bounded by $p+9$ and,
by \cite[Theorem 3.3]{OseledetsMatAndInv}, $M_L^{-1}$ has QTT-ranks
bounded by $5$. Then, by \cite[Section 4.3]{OseledetsTTdec}, their product
$M_L^{-1}v$ has QTT-ranks bounded by the product of the QTT-ranks; that
is, bounded by $5(p+9)$. The $\mathbb{P}_{1}$-FE coefficient vector
of $\PiLtwo \udp$ is given by $ M_L^{-1} v$, and hence has QTT-ranks bounded by $5(p+9)$.
\end{proof}
\subsection{\textit{A priori} error analysis}
\label{sec:error-estimate}
We now turn our focus to estimating the error $\| \PiLtwo
\udp-\ud\|_{\delta}$. Our goal (i.e., the result of Theorem
\ref{QTTapprox} below) is to prove that there exists $C>0$ such that for all $L
\in \mathbb{N}$ and all $0<\delta<1$,
there holds
\begin{equation*}
  || \PiLtwo\udL - \ud ||_{\delta} \leq C
\min\left(\sqrt{h},\frac{h}{\sqrt{\delta}}\right),
\end{equation*}
where $h = 1/(2^L+1)$.
First, by the triangle inequality and stability \eqref{eq:PiLtwostability},
\begin{equation}\label{fullEstimate}
  \begin{aligned}
  || u_{\delta} - \PiLtwo \udL ||_{\delta}
  & \leq
  || u_{\delta} - \Pi_L u_{\delta} ||_{\delta}
  + ||\PiLtwo \left(\Pi_L\ud-\ud \right) ||_{\delta}
  + ||\PiLtwo \left(u_{\delta}-\udL \right) ||_{\delta} 
  \\ & \leq (1+\Cstab) || u_{\delta} - \Pi_L u_{\delta} ||_{\delta} + \Cstab\|\ud-\udL\|_{\delta}.
  \end{aligned}
\end{equation}
The second term at the right hand side of the equation above can be estimated
according to Proposition \ref{hpFEMTheorem}.
It remains to
bound the first term on the right hand side of equation \eqref{fullEstimate}.
\subsubsection{Error analysis for $\mathbb{P}_1$ finite elements}
For ease of presentation, we
assume that $\alpha_0 = \alpha_1 = 0$ in the following results and then remove this
assumption in Theorem \ref{QTTapprox}, our main theorem.

We recall $\delta$-dependent upper bounds for the Sobolev norms of the solution
to \eqref{DiffEq}, under different regularity assumptions on the right hand side.
The first result
considers the weak assumption of $f \in L^{2}(I)$.
\begin{proposition}[{\cite[Lemma 2.1]{SingPerturbProb}}]\label{solProp}
  Let $0< \delta < 1$, $u_{\delta} \in H^{2}(I)$ be the solution of equation
\eqref{DiffEq} with $\alpha_0 = \alpha_1 = 0$, $f \in L^{2}(I)$ and $c\in
L^{\infty}(I)$, $c(x)\geq
c_{\min{}} > 0$ for all $x\in \overline{I}$. There exists
$C>0$, independent of $\delta$ and $f$, such that 
\begin{equation*}
 \delta^2\|\ud \|_{H^2(I)}+ \delta \| \ud\|_{H^1(I)} +\| \ud \|_{L^2(I)}  \leq C\|f\|_{L^2(I)}.
\end{equation*}
\end{proposition}
If $f \in H^{1}_{0}$, then the following stronger result holds:
\begin{lemma}[{\cite[Lemma A.2]{SingPerturbProb}}]\label{ELemma}
  Let $0< \delta
< 1$, $u_{\delta} \in H^{2}(I)$ be the solution of equation \eqref{DiffEq} with
$\alpha_0 =\alpha_1=0$, $f \in H^{1}_{0}(I)$ and $c\in W^{1, \infty}(I)$, $c(x)\geq
c_{\min{}} > 0$ for all $x\in \overline{I}$. There exists $C>0$, independent
of $\delta$ and $f$, such that
\begin{equation*}
  \delta ||u_{\delta}||_{H^{2}(I)} + ||u_{\delta}||_{H^{1}(I)}
\leq C ||f||_{H^{1}(I)}.
\end{equation*}
\end{lemma}
The following result on the $\mathbb{P}_{1}$-Galerkin
projection will also be needed.
\begin{proposition}[{\cite[Lemma 4.1]{SingPerturbProb}}]\label{ProjectionProp}
There exists a constant $C > 0$, such that for all $L\in \mathbb{N}$ and $v \in
H^{2}(I)$, writing $h = 1/(2^L+1)$,
\begin{equation*}
  || (\Pi_{L} v)' - v' ||_{L^{2}(I)} \leq C
  \begin{cases} 
    ||v||_{H^{1}(I)} \\
    h ||v||_{H^{2}(I)} ,
  \end{cases}
\end{equation*}
and
\begin{equation*}
  || \Pi_{L} v - v ||_{L^{2}(I)} \leq C
  \begin{cases} h
    ||v||_{H^{1}(I)} \\
     h^{2} ||v||_{H^{2}(I)} .
  \end{cases}
\end{equation*}
\end{proposition}
We now introduce interpolation spaces between $L^{2}(I)$ and
$H^{1}_{0}(I)$.
\begin{definition}\label{interpolantSpace}
  For $0 < \theta < 1$, we define the
interpolation space $H^{\theta,\infty}(I)$ between $L^{2}(I)$ and $H^{1}_{0}(I)$
as
\begin{equation*}
  H^{\theta,\infty}(I) \coloneqq \{ v \in H^{1}(I) \ : \
||v||_{\theta,\infty} < \infty \},
\end{equation*}
where
\begin{equation*}
  ||v||_{\theta,\infty} \coloneqq \sup_{t > 0}
\frac{K(t,f)}{t^{\theta}},\quad \text{and}\quad K(t,f) \coloneqq
\inf_{\substack{v = v_{0} + v_{1} \\ v_{0} \in L^{2}(I),\ v_{1} \in
H^{1}_{0}(I)}} ||v_{0}||_{L^{2}(I)} + t ||v_{1}||_{H^{1}(I)}.
\end{equation*}
\end{definition}
See, e.g., \cite{Brenner2008,InterpolationSpaces} for more
details on interpolation spaces. The following result implies that the right
hand side $f$ of equation \eqref{DiffEq} is contained in $H^{1/2,\infty}(I)$.
\begin{lemma}[{\cite[Chapter 2, Section 5, Lemma
    5.2]{Lions}}]\label{interpolantIncl}
  $H^{1}(I) \subset H^{1/2,\infty}(I)$ with
continuous inclusion.
\end{lemma}
The following proposition establishes the desired error estimates on
the $\mathbb{P}_{1}$-FE solution of problem \eqref{DiffEq} with homogeneous
Dirichlet boundary conditions.
\begin{proposition}\label{errorEstimates}
  Let $0< \delta < 1$ and let
$u_{\delta} \in H^{2}(I)$ be the solution of equation \eqref{DiffEq} with
$\alpha_0=\alpha_1 = 0$ and $f$ and $c$ subject to \eqref{eq:hypfc}, Then there exists $C>0$, independent of $\delta$ such that,
for all $L\in \mathbb{N}$, writing $h = 1/(2^L+1)$,
\begin{enumerate}
    \item If $f \in L^{2}(I)$, then
      \begin{equation*}
        \begin{aligned}
          ||\Pi_L u_{\delta} - u_{\delta} ||_{L^{2}(I)}
          & \leq C \begin{cases}
            ||f||_{L^{2}(I)}, & \text{for } h \geq \delta \\
            \dfrac{h^{2}}{\delta^{2}} ||f||_{L^{2}(I)}, & \text{for } h \leq \delta
              \end{cases}
              \\
              \delta |\Pi_L u_{\delta} - u_{\delta} |_{H^{1}(I)}
              & \leq C \begin{cases} 
                ||f||_{L^{2}(I)}, & \text{for } h \geq \delta \\
                \dfrac{ h}{\delta}
||f||_{L^{2}(I)}, & \text{for } h \leq \delta \end{cases}
    \end{aligned}
    \end{equation*}
    \item If $f \in H^{1/2,\infty}(I)$, then
      \begin{equation*}
        \begin{aligned}
          || \Pi_L u_{\delta} - u_{\delta} ||_{L^{2}} & \leq C
          \begin{cases}  \sqrt{h} ||f||_{H^{1/2,\infty}(I)}, & \text{for } h \geq \delta \\
            \dfrac{ h^{2}}{\delta^{3/2}} ||f||_{H^{1/2,\infty}(I)}, & \text{for } h \leq \delta
          \end{cases}
          \\
          \delta |\Pi_L u_{\delta} - u_{\delta} |_{H^{1}(I)} & \leq C
          \begin{cases} \sqrt{h} ||f||_{H^{1/2,\infty}(I)}, & \text{for } h \geq \delta \\
            \dfrac{h}{\delta^{1/2}} ||f||_{H^{1/2,\infty}(I)}, & \text{for } h \leq \delta
          \end{cases}
    \end{aligned}
    \end{equation*}
    \item If $f \in H^{1}_{0}(I)$, then
      \begin{equation*}
        \begin{aligned}
          || \Pi_L u_{\delta} - u_{\delta} ||_{L^{2}(I)} & \leq C
          \begin{cases} h ||f||_{H^{1}(I)}, & \text{for } h \geq \delta \\
            \dfrac{h^{2}}{\delta} ||f||_{H^{1}(I)}, & \text{for } h \leq \delta
          \end{cases}
          \\
          \delta |\Pi_L u_{\delta} - u_{\delta} |_{H^{1}(I)} & \leq C h ||f||_{H^{1}(I)}
    \end{aligned}
    \end{equation*}
  \end{enumerate}
  In particular, if $f \in H^{1/2,\infty}(I)$ then
  \begin{equation*}
    || \Pi_L u_{\delta} - u_{\delta} ||_{\delta} \leq C \min
\left( \sqrt{h}, \frac{h}{\sqrt{\delta}} \right).
\end{equation*}
\end{proposition}
\begin{proof}  The $L^{2}$-norm estimates are
proved in \cite[Theorem A.1]{SingPerturbProb} and we adapt their strategy to
prove the corresponding $H^{1}$-seminorm estimates. We begin with proving
Assertion 1 and Assertion 3 and then use an interpolation technique to establish
Assertion 2.

Assertion 1: Suppose that $f \in L^{2}(I)$. The $H^{1}$-seminorm estimates
follow from Proposition \ref{solProp} and Proposition \ref{ProjectionProp}
\begin{equation*}
  \delta |\Pi_L \ud - \ud |_{H^{1}(I)} \leq C
  \begin{cases}
    \delta ||\ud||_{H^{1}(I)}, & \text{if }h \geq \delta \\
    \delta h ||\ud||_{H^{2}(I)}, & \text{if }h \leq \delta
  \end{cases}
  \leq C
  \begin{cases} ||f||_{L^{2}(I)}, & \text{if }h \geq \delta \\
  \dfrac{h}{\delta} ||f||_{L^{2}(I)}, & \text{if }h \leq \delta.
\end{cases}
\end{equation*}

Assertion 3: Suppose that $f \in H^{1}_{0}(I)$. The $H^{1}$-seminorm estimates
follow from Proposition \ref{ProjectionProp} and Lemma \ref{ELemma}
\begin{equation*}
  \delta |\Pi_L \ud - \ud |_{H^{1}(I)} \leq C \delta h
||\ud||_{H^{2}(I)} \leq C h ||f||_{H^{1}(I)}.
\end{equation*}

Assertion 2: Suppose that $f \in H^{1/2,\infty}(I)$. Decompose $f = f^{0}
+f^{1}$ for any $f^{0} \in L^{2}(I)$ and for any $f^{1} \in H^{1}_{0}(I)$. Let
$\ud^{i}$, for $i=0,1$, be the solution of equation \eqref{DiffEq} with
$\alpha_0 = \alpha_1=0$ but with right hand side equal to $f^{i}$. By linearity and by
the triangle inequality we have
\begin{multline*}
  \delta | \Pi_L \ud - \ud |_{H^{1}(I)} \leq \delta |\Pi_L \ud^{0}
- \ud^{0}|_{H_{1}(I)} + \delta | \Pi_L \ud^{1} - \ud^{1}|_{H^{1}(I)} \\
\leq C \begin{cases} ||f^{0}||_{L^{2}(I)} + h ||f^{1}||_{H^{1}(I)}, &
\text{for } h \geq \delta \\ \dfrac{h}{\delta} \left( ||f^{0}||_{H^{1}(I)} +
\delta ||f^{1}||_{H^{1}(I)} \right) , & \text{for } h \leq \delta \end{cases}
\end{multline*}
and by taking the infimum over all such decompositions of $f$ we
obtain that
\begin{equation*}
  \delta | \Pi_L \ud - \ud |_{H^{1}(I)} \leq C \begin{cases} 
K(h,f) & \text{for } h \geq \delta \\ \dfrac{h}{\delta} K(\delta,f), &
\text{for } h \leq \delta.
\end{cases}
\end{equation*}
Since
\begin{equation*}
  ||f||_{1/2,\infty} \geq \max\left(\dfrac{K(h,f)}{\sqrt{h}},
\dfrac{K(\delta,f)}{\sqrt{\delta}}  \right),
\end{equation*}
we obtain the desired estimate for the $H^{1}$-seminorm
\begin{equation*}
  \delta | \Pi_L \ud - \ud |_{H^{1}(I)} \leq C \begin{cases} 
\sqrt{h} ||f||_{H^{1/2,\infty}(I)}, & \text{for } h \geq \delta \\ \dfrac{
h}{\sqrt{\delta}} ||f||_{H^{1/2,\infty}(I)} & \text{for } h \leq
\delta. \end{cases}
\end{equation*}
\end{proof}
\subsubsection{Error estimate for the low-rank QTT approximation}
We are now in a position to prove our main theorem, Theorem
\ref{QTTapprox}. In particular, Theorem \ref{QTTapprox} establishes existence of
low-rank QTT-approximations converging exponentially fast to the solution
$u_{\delta}$ of problem \eqref{DiffEq} with respect to the number of degrees of
freedom and uniformly in $0 < \delta < 1$.
\begin{theorem}\label{QTTapprox}
There
exist $C_{1},b_{1} > 0$ such that, for all $0<\delta<1$,
if $\ud$ is the solution to \eqref{DiffEq} under the hypotheses \eqref{eq:hypfc}, and
for all $L\in \mathbb{N}$, if $\udL$ is the approximation given by Proposition
\ref{hpFEMTheorem}, then
\begin{equation*}
  || \PiLtwo \udL - u_{\delta}||_{\delta} \leq
C_{1} \left( e^{-b_{1} L} + \min \left(\sqrt{h},\frac{h}{\sqrt{\delta}} \right)
\right),
\end{equation*}
with $h = 1/(2^L+1)$. Furthermore, $\PiLtwo \udL$ admits a QTT decomposition
with respect to $\TunifintL$ with
QTT-ranks of order $\cO(L)$ and with number of parameters of order $\Ndof = \cO(L^{3})$.

With respect to the number of parameters $\Ndof$ of the QTT representation of $\PiLtwo
\udL$, the above inequality reads
\begin{equation}\label{dofEstimate}
  || \PiLtwo\udL - u_{\delta}||_{\delta} \leq C_{2} \exp\left(-b_{2} \Ndof^{1/3}  \right),
\end{equation}
with $C_2, b_2>0$, independent of $\delta$ and $L$.
\end{theorem}
\begin{proof} As we do not assume homogeneous boundary conditions on $\ud$, we introduce
\begin{equation*}
  v_{\delta} \coloneqq u_{\delta} - g,
\end{equation*}
where $g(x) \coloneqq \alpha_0 + (\alpha_1 - \alpha_0)x$. Then $v_{\delta}$ satisfies the following modified
differential equation with homogeneous boundary conditions
\begin{equation*}
  \begin{aligned}\label{DiffEqMod}
    - \delta^{2} v_{\delta}'' + v_{\delta} &= f- cg, \text{ in } I
    \\
    v_{\delta}(0) =v(1) &= 0.
\end{aligned}
\end{equation*}
Clearly, the regularity assumptions on $f$ from Section
\ref{statement} also hold for $f-cg$. In particular, by Lemma
\ref{interpolantIncl} and by the assumption \eqref{eq:hypfc}, we have that
$f-cg \in H^{1/2,\infty}(I)$. Hence, the last assertion of Proposition
\ref{errorEstimates} is applicable to $v_{\delta}$. Thus,
\begin{equation*}
  || \Pi_L v_{\delta} - v_{\delta} ||_{\delta} \leq C \min
\left( \sqrt{h},\frac{h}{\sqrt{\delta}} \right).
\end{equation*}
Moreover, $\Pi_L g = g$, as $g \in V^{L}$ ($g$ is affine), and
$\Pi_L : H^1(I) \to V^{L}$ is a projection. Therefore,
\begin{equation*}
  || \Pi_{L} u_{\delta} - u_{\delta}||_{\delta} = || \Pi_{L}
(u_{\delta}-g)-(u_{\delta}-g)||_{\delta}.
\end{equation*}
Since $v_{\delta} = u_{\delta} - g$, we obtain
\begin{equation}
  \label{eq:PiL-error}
  || \Pi_{L} u_{\delta} - u_{\delta}||_{\delta} \leq C \min
\left( \sqrt{h},\frac{h}{\sqrt{\delta}} \right).
\end{equation}
As we have already stated in equation \eqref{fullEstimate},
\begin{align*}
  || u_{\delta} - \PiLtwo \udL ||_{\delta}
  & \leq
  || u_{\delta} - \Pi_L u_{\delta} ||_{\delta}
  + ||\PiLtwo \left(\Pi_L\ud-\ud \right) ||_{\delta}
  + ||\PiLtwo \left(u_{\delta}-\udL \right) ||_{\delta} 
  \\ & \leq (1+\Cstab) || u_{\delta} - \Pi_L u_{\delta} ||_{\delta} + \Cstab\|\ud-\udL\|_{\delta}.
  \end{align*}
  Then, inequality \eqref{eq:PiL-error} and Proposition \ref{hpFEMTheorem} yield the desired error
estimate
\begin{equation*}
  || \PiLtwo \udL - u_{\delta}||_{\delta} \leq C
\left( e^{-b L} + \min \left(\sqrt{h},\frac{h}{\sqrt{\delta}} \right) \right).
\end{equation*}
Let now $\vdL$ be the approximation to $v_\delta$  provided by
Proposition \ref{hpFEMTheorem}, with $p=L$.
By Proposition \ref{PropQTTranks}, $\PiLtwo \vdL$ has a QTT representation with
respect to $\TunifintL$ with QTT-ranks of order $\cO(L)$. Then,
\begin{equation*}
  \PiLtwo \udL = \PiLtwo\vdL + g,
\end{equation*}
and $g$ is an affine function in $I$, hence it has a QTT representation with
respect to $\TunifintL$ with rank bounded by $L$.
The QTT-rank of the sum of two QTT-formatted vectors is bounded by the sum of
their ranks \cite{OseledetsTTdec}. Therefore, $\PiLtwo \udL$ admits a QTT decomposition with respect to
$\TunifintL$ with ranks $\cO(L)$, and the number of parameters of the
QTT decomposition of $\PiLtwo\udL$ is of order $\cO(L^{3})$. The
last assertion now follows directly, as there exist $\widetilde{C},
\widetilde{b}>0$ independent of $L$, such that
\begin{equation*}
  \min \left( \sqrt{h}, \frac{h}{\sqrt{\delta}} \right) \leq 
\sqrt{h} \leq  2^{\frac{1}{2}(1-  L) } \leq \widetilde{C} \exp\left(-\widetilde{b} \sqrt[3]{\Ndof}\right).
\end{equation*}
\end{proof}

\section{QTT-formatted computations and preconditioning}
\label{numerics}

As in previous sections, we consider equation~\eqref{DiffEq} and approximate the solution in low-order
$\mathbb{P}_1$ finite element space. 
To solve the arising linear system \eqref{sysM}, 
one can assemble the system $A_L$ and approximate the right-hand side $f_L$ directly in the QTT format.
In particular, the system matrix and the load vector can be represented in approximate low rank QTT formulation, through a combination of exact representations~\cite{kazeev2012low} and adaptive sampling~\cite{oseledets2010tt}
The system of linear equations can then be approximately solved using well-established optimization-based algorithms such as the alternating minimal energy solver (AMEn)~\cite{AMEnLinearSystems}.
Nevertheless, recent studies~\cite{chertkov2016robust,BachmayrStability,RakhubaRobustSolver} have indicated that stability issues may occur when using this approach for large $L$.

To overcome this problem, a BPX-type preconditioner in the QTT format was proposed in~\cite{BachmayrStability}.
Instead of the standard left BPX preconditioner, the authors proposed a symmetric two-sided version that ensures robustness in QTT format.
The approach is applicable for a wide range of elliptic PDEs, but does not provide $\delta$-robustness, so we have modified the preconditioner for our purposes.

Let us briefly introduce the original preconditioner and its modification.
The implementation details are presented in Appendix~\ref{sec:assembly_prec}.
For all $\ell\in\mathbb{N}$, we introduce the piecewise linear basis functions $\hat\phi_{\ell,j}$ that satisfy
\[
	\hat\phi_{\ell,j} (2^{-\ell}i) = 2^{\ell/2} \delta_{ij}, \qquad \forall i, j = 0,\dots,2^\ell.
\]
Let then $\hat P_{\ell, L} \in\mathbb{R}^{2^L \times 2^\ell}$ be the matrix
associated to the canonical injection from $\mathrm{span}(\hat\phi_{\ell,j})$
into $\mathrm{span}(\hat\phi_{L,j})$. We introduce the symmetric preconditioner
\begin{equation}\label{eq:bpx_prec}
	\widetilde C_L = \sum_{\ell=0}^{L} 2^{-\ell} P_{\ell,L} P_{\ell,L}^{T}
\end{equation}
so that the preconditioned system is $\widetilde C_L A_L \widetilde C_L
\bar{u}_{L} = \widetilde C_L f_L$.
This preconditioner was first introduced in~\cite{BachmayrStability} and is a symmetric version of the classic BPX preconditioner~\cite{bramble1990parallel}.
The modification to $\widetilde C_L$ that we propose to use for singularly perturbed problems reads
instead
\begin{equation}\label{eq:prec_perturbed}
	 C_{L} = \sum_{\ell=0}^{L} \bpxcoef P_{\ell,L} P_{\ell,L}^{T}
\end{equation}
where $\bpxcoef$ is chosen to ensure independence of the condition number in
terms of $\delta$. Specifically, we choose $\bpxcoef = \min (2^{-\ell}\delta^{-1},1)$.
Note that $\bpxcoef = (1 + \delta^2 2^{2\ell})^{-1}$ is used in~\cite{bramble2000computational} for one-sided preconditioner, so the square root is applied for its two-sided version;
when $\delta \gg 2^{-\ell}$ and $\delta \ll 2^{-\ell}$, our choice has the same
behavior as the choice $\bpxcoef = (1 + \delta^2 2^{2\ell})^{-1/2}$.
Details of the assembly of~\eqref{eq:prec_perturbed} and of the preconditioned matrix $C_{L} A_L C_{L}$ are deferred to Section~\ref{sec:assembly_prec}.

\section{Numerical experiments}
\label{sc:TT-numerics}
In this section we
consider the following instance of problem \eqref{DiffEq}:
\begin{equation}
  \begin{aligned}\label{diffEq}
    &- \delta^{2} u_{\delta}'' + u_{\delta} = 0 \text{ in } I, \\
    &u_{\delta}(0) = 0,\, u_{\delta}(1)=1,
\end{aligned}
\end{equation}
where $0<\delta<1$ is the perturbation parameter.
The corresponding Dirichlet-Neumann problem as in equation
\eqref{DNdiffEqMod} then reads
\begin{equation}
\begin{aligned} - \delta^{2} v_{\delta}''(x) + v_{\delta}(x) &= -x \text{ in }I,
\\ v_{\delta}(0) = v_{\delta}'(1)&=0.
\end{aligned}
\end{equation}
We use the publicly available TT-Toolbox\footnote{\url{https://github.com/oseledets/TT-Toolbox}} as
a basis for our experiments and to perform fundamental operations in QTT format.
Problem \eqref{diffEq} has the
exact solution
\begin{equation}\label{exactSol}
  u_{\delta}(x) =
\frac{e^{\frac{x}{\delta}}-e^{-\frac{x}{\delta}}}{e^{\frac{1}{\delta}}-e^{-\frac{1}{\delta}}}
=
\frac{e^{\frac{x-1}{\delta}}-e^{-\frac{x+1}{\delta}}}{1-e^{-\frac{2}{\delta}}}.
\end{equation}

Let $\udqttvec\in \mathbb{R}^{2^L}$ be the vector corresponding to the 
QTT-compressed numerical solution of problem \eqref{diffEq}, and let
\begin{equation*}
  \udqtt(x) = \sum_{i=1}^{2^L} \left( \udqttvec \right)_i \phi_i(x) , \qquad \forall x\in (0,1),
\end{equation*}
where $\{\phi\}_{i=1}^{2^L}$ is the Lagrange basis associated with $\TtildeL$.
The error $e_{\delta}$ is computed as
\begin{equation}\label{erroDef}
  e_{\delta} \coloneqq \| \ud- \udqtt\|_{\delta}.
\end{equation}
In practice, due to its size, the vector $\udqttvec$ is not explicitly
treatable, and the computation of \eqref{erroDef} is done in QTT format.

We use a DMRG solver for all algebraic equations in our numerical simulations
and we denote by
$\epstol > 0$ tolerance level for termination of the DMRG iterations,
measured by the Frobenius norm of the residual.
We also use the same value $\epstol$ as an accuracy parameter for rounding of
QTT-formatted tensors;
see  \cite{OseledetsTTdec} for more details on rounding in the
TT-format. The values of $\epstol$ used in the computations will be specified on
a case-by-case basis.
\subsection{Non-preconditioned system}
\begin{figure} \centering
\includegraphics[width=0.5\textwidth]{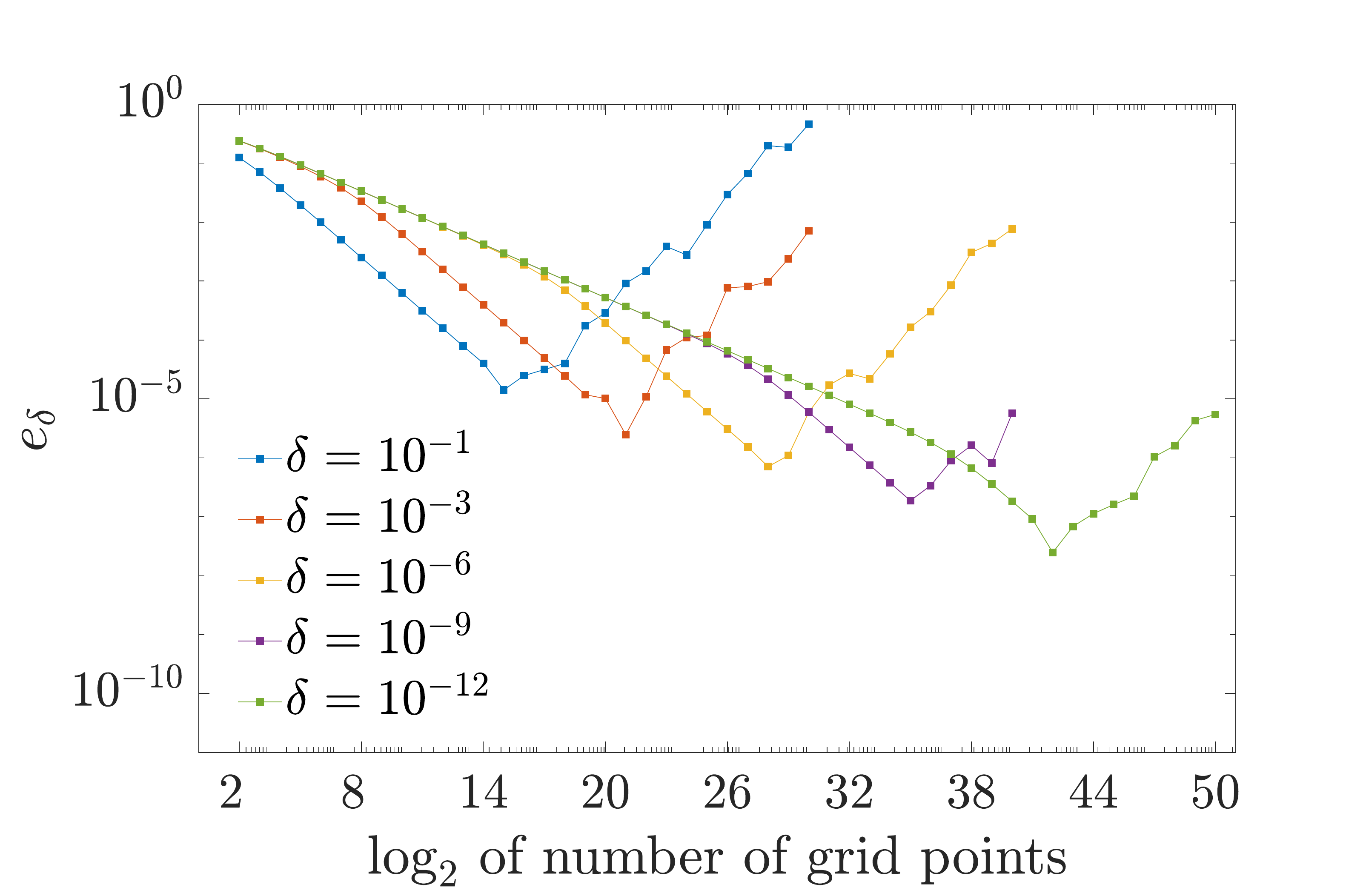}
\caption{Error $e_{\delta}$, obtained without preconditioner, as a function of the
logarithm of the number of grid points, for $\delta
\in \{ 10^{-1}, 10^{-3}, 10^{-6}, 10^{-9}, 10^{-12} \}$ and with $\epstol =
10^{-10}$.}
\label{fig:WO_WO_E_norm_8_10}
\end{figure}

We first consider the non-preconditioned, direct application of the DMRG
solver.  We assemble the FE matrices in
the QTT format and then solve the resulting system with the DMRG solver.
Figure \ref{fig:WO_WO_E_norm_8_10} displays the
error $\ed$ obtained for varying values of $\delta$ and with $\epstol = 10^{-10}$.
We observe instabilities for all
considered values of $\delta$. Choosing a different
value of $\epstol$ did not influence the results.

We remark that we see three difference regions in the behavior of the error and that
the first two regions correspond to the expected theoretical
rates of convergence of the non compressed system, obtained in Proposition \ref{errorEstimates}.
Namely, denoting $h = 1/(2^L+1)$,
\begin{itemize}
  \item for $L\lesssim \left|\log_2\delta  \right|$ (i.e., $h > \delta$), we observe convergence of order $\cO(h^{1/2})$,
  \item for $L\gtrsim \left| \log_2\delta \right|$ (i.e., $h < \delta$), we observe convergence of order $\cO(h)$,
\end{itemize}
while for large values of $L$ the approximation is unstable.

\subsection{Preconditioned system}
\begin{figure} \centering
\begin{subfigure}{.5\textwidth} \centering
\includegraphics[width=\textwidth]{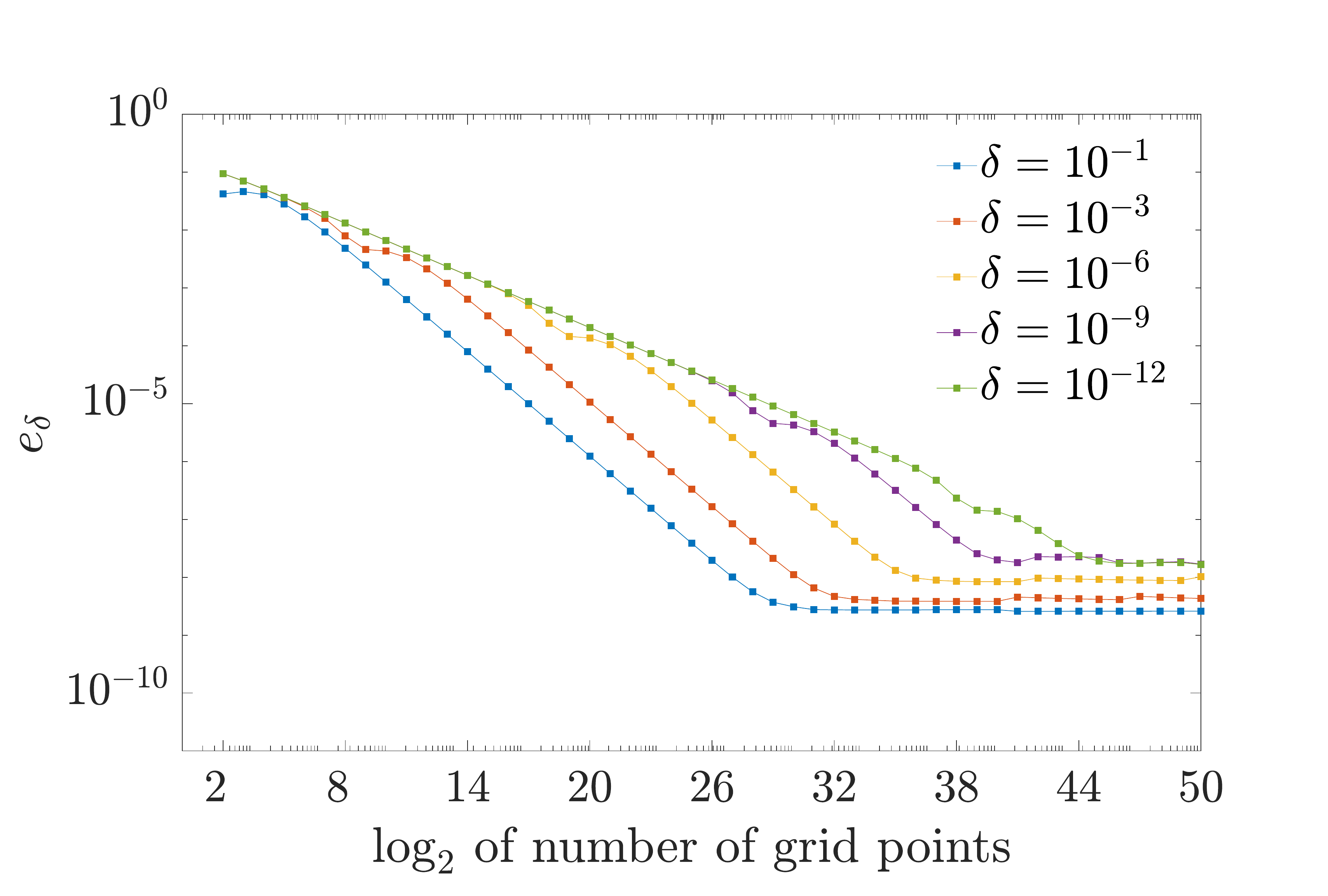}
\caption{$\epstol = 10^{-8}$} \label{fig:W_W_E_norm_8_10_new_a}
\end{subfigure}%
\begin{subfigure}{.5\textwidth} \centering
\includegraphics[width=\textwidth]{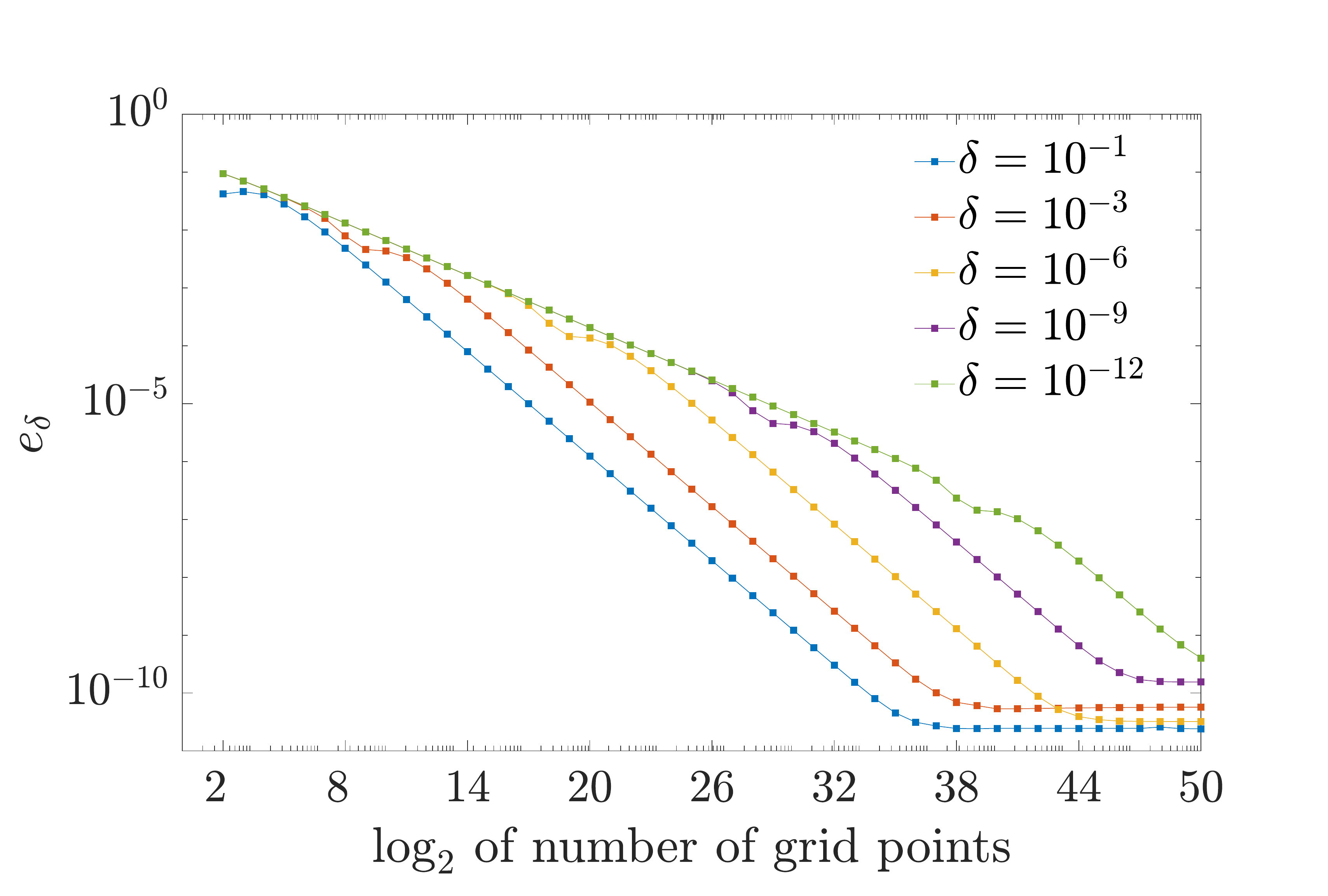}
\caption{$\epstol = 10^{-10}$} \label{fig:W_W_E_norm_8_10_new_b}
\end{subfigure} \\
\begin{subfigure}{.5\textwidth} \centering
\includegraphics[width=\textwidth]{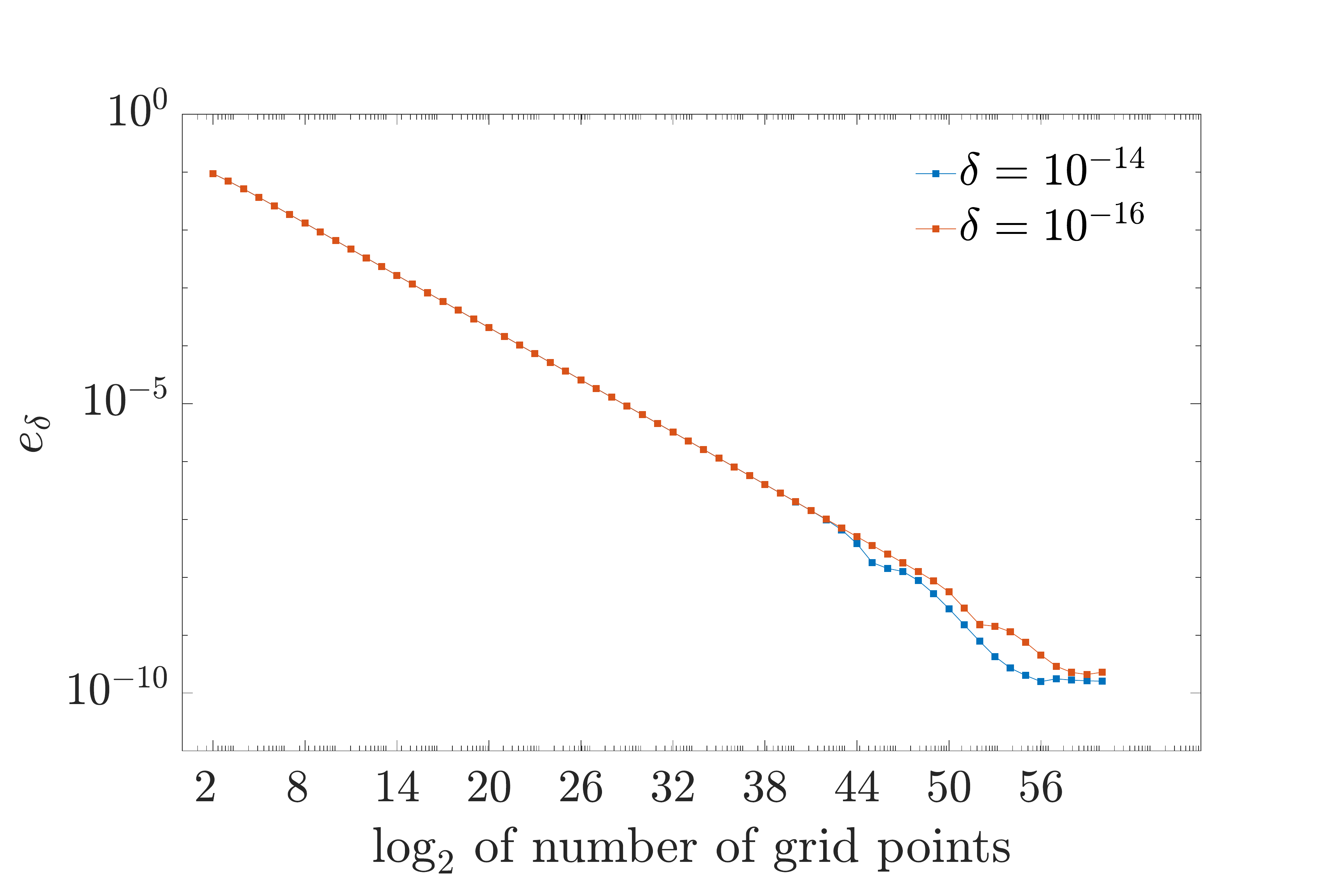}
\caption{$\epstol = 10^{-10}$} \label{fig:W_W_E_norm_8_10_new_c}
\end{subfigure}%
\caption{Error $e_{\delta}$ as a function of the
logarithm of the number of grid points for computations with preconditioner and $\delta \in
\{ 10^{-1}, 10^{-3}, 10^{-6}, 10^{-9}, 10^{-12} \}$.} \label{fig:W_W_E_norm_8_10_new}
\end{figure}
The numerical experiments in the previous
section illustrate the need for preconditioning the QTT-compressed version of
the algebraic equation \eqref{sysM} obtained from the $\mathbb{P}_{1}$-FE
method.
In this section we present the results from the simulations of
the preconditioned and modified system, as described in Section \ref{numerics}. 
The simulations for varying values of $\delta$ and $\epstol$ are shown in Figure \ref{fig:W_W_E_norm_8_10_new}. We observe no instabilities, as were
present in the non-preconditioned system.
We also observe the same asymptotic rates observed for the stable region of the
non-preconditioned solver and theoretically predicted in Proposition \ref{errorEstimates};
in this case, though, for large values of $L$ the errors reach a plateau. The comparison of Figures
 \ref{fig:W_W_E_norm_8_10_new_a} and \ref{fig:W_W_E_norm_8_10_new_b} shows that
 this depends on the choice of $\epstol$.
\begin{figure} \centering
\includegraphics[width=.5\textwidth]{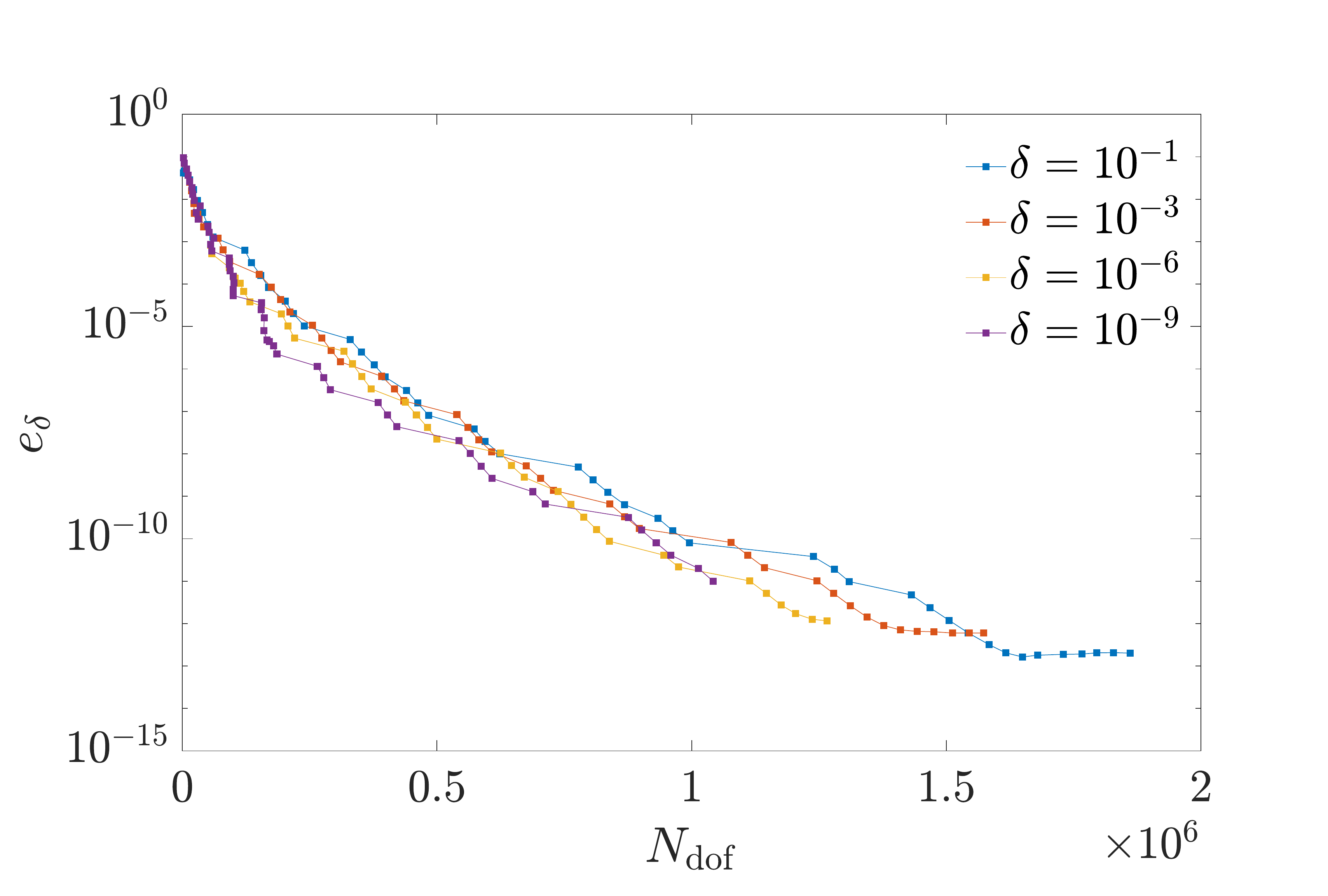}
\caption{Error $e_{\delta}$ as function of the number of
degrees of freedom for computations with preconditioner, adaptive choice of
$\epstol$, and for $\delta \in \{ 10^{-1}, 10^{-3},
10^{-6}, 10^{-9} \}$.}
\label{fig:W_W_E_norm_dof_new}
\end{figure}

The error plot we have considered so far have all considered the variation of
the error for a fixed truncation value $\epstol$. Estimate \eqref{dofEstimate}
of Theorem \ref{QTTapprox}, though, gives a theoretical exponential rate of
convergence, with constants independent of $\delta$. This can be realized, in
practice, by choosing adaptively the truncation parameter $\epstol$, so that,
for each $\delta$ and each $L$, the truncation error is of the same order of
magnitude as the finite element error.
Figure \ref{fig:W_W_E_norm_dof_new} displays the error $e_{\delta}$ as a
function of the number of degrees of freedom for solutions obtained with the
adaptive choice of the
accuracy parameter $\epstol$ outlined. Specifically, for each $\delta$ and for each $L$,
we choose the biggest $\epstol$ such that the error obtained is at
most $ 10\% $ larger than the error $e_{\delta}$ obtained with $\epstol = 10^{-12}$.

\begin{figure} \centering
\begin{subfigure}{.5\textwidth} \centering
\includegraphics[width=\textwidth]{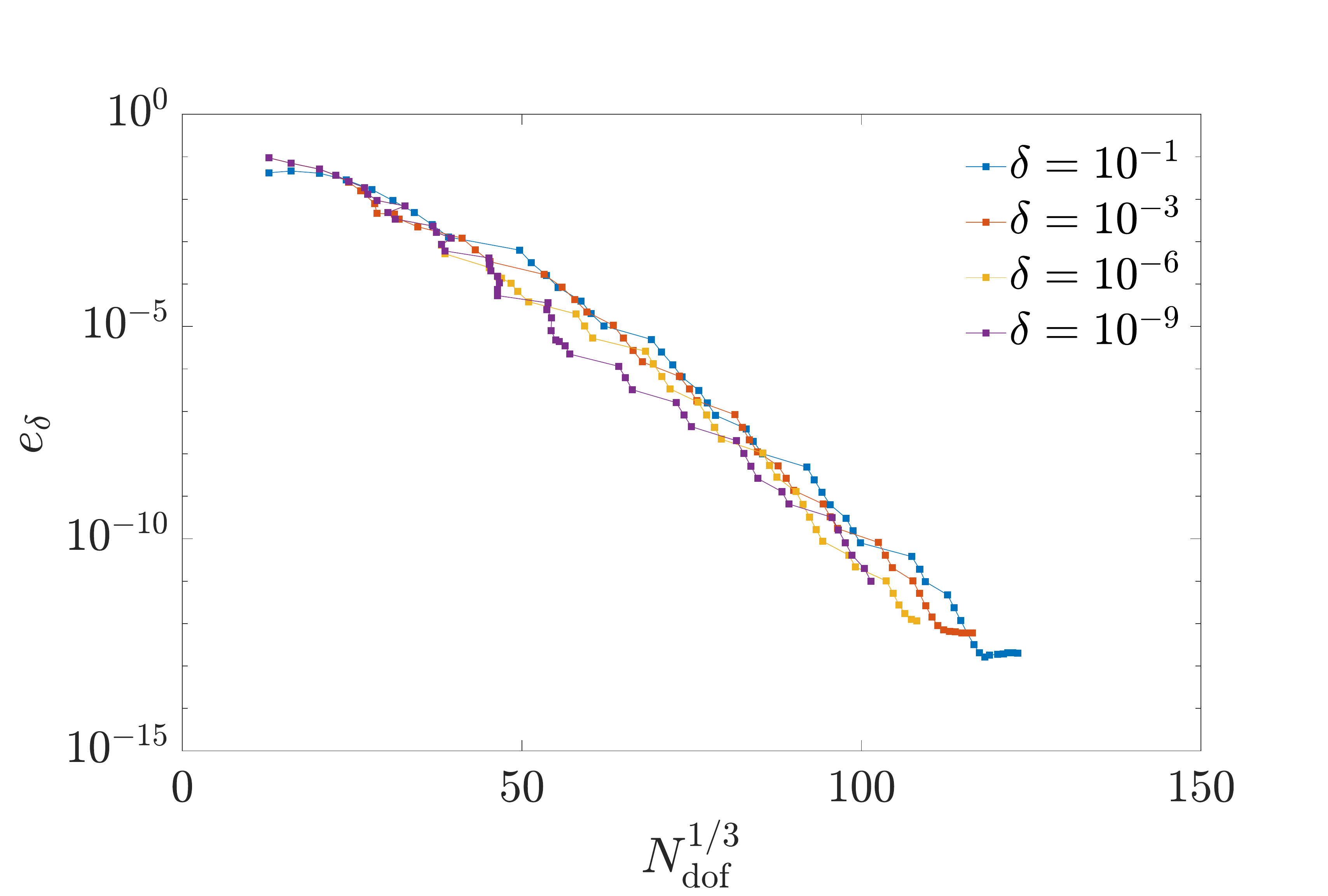}
\end{subfigure}%
\begin{subfigure}{.5\textwidth} \centering
\includegraphics[width=\textwidth]{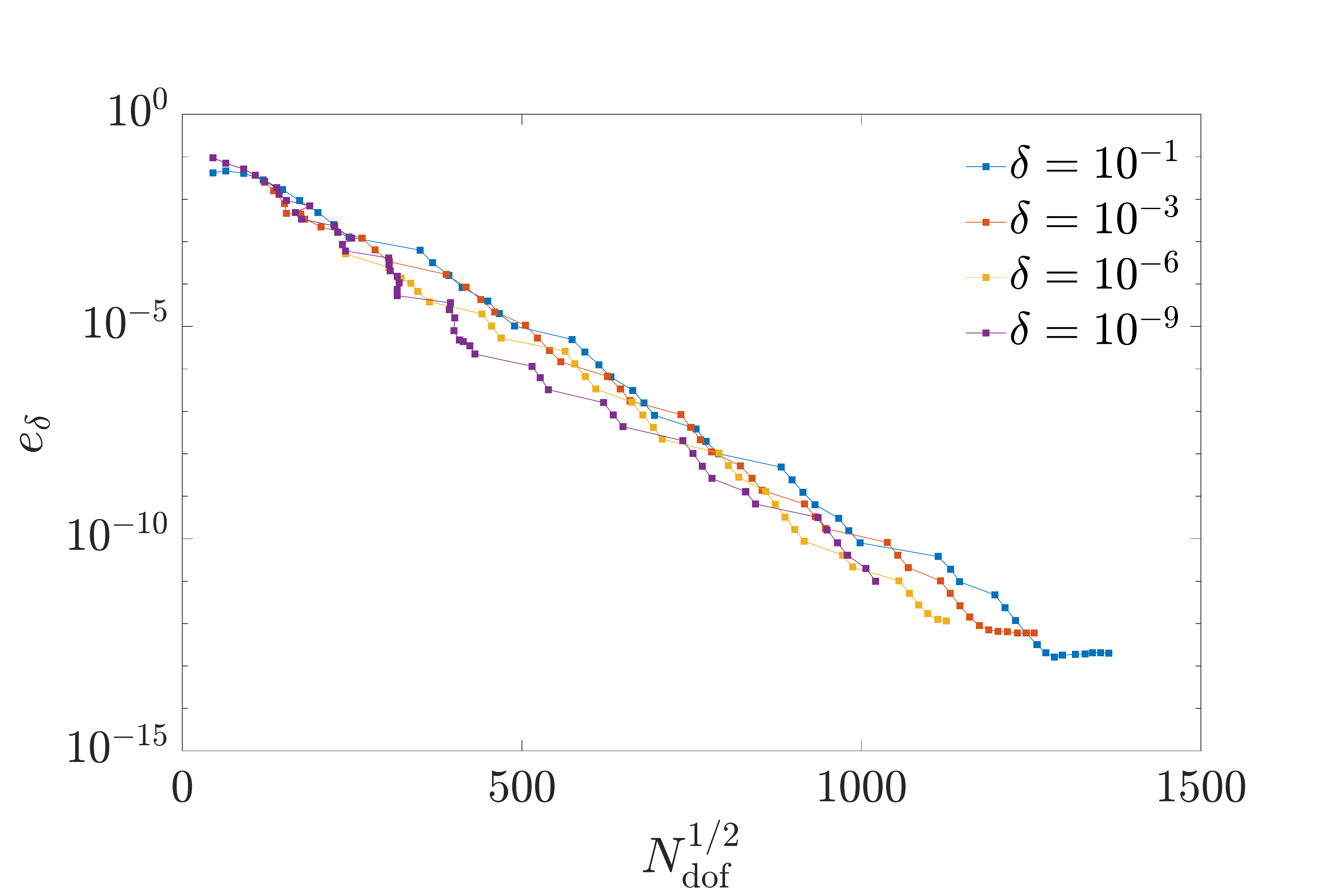}
\end{subfigure}
\caption{Error $e_{\delta}$ as a function of the cube root, left,
  and the square root, right, of the number of degrees of freedom
  for computations with preconditioner, adaptive choice of
$\epstol$, and $\delta \in \{ 10^{-1}, 10^{-3}, 10^{-6}, 10^{-9}\}$.}
\label{fig:W_W_E_norm_dof_2_new}
\end{figure}

\begin{figure} \centering
\includegraphics[width=.5\textwidth]{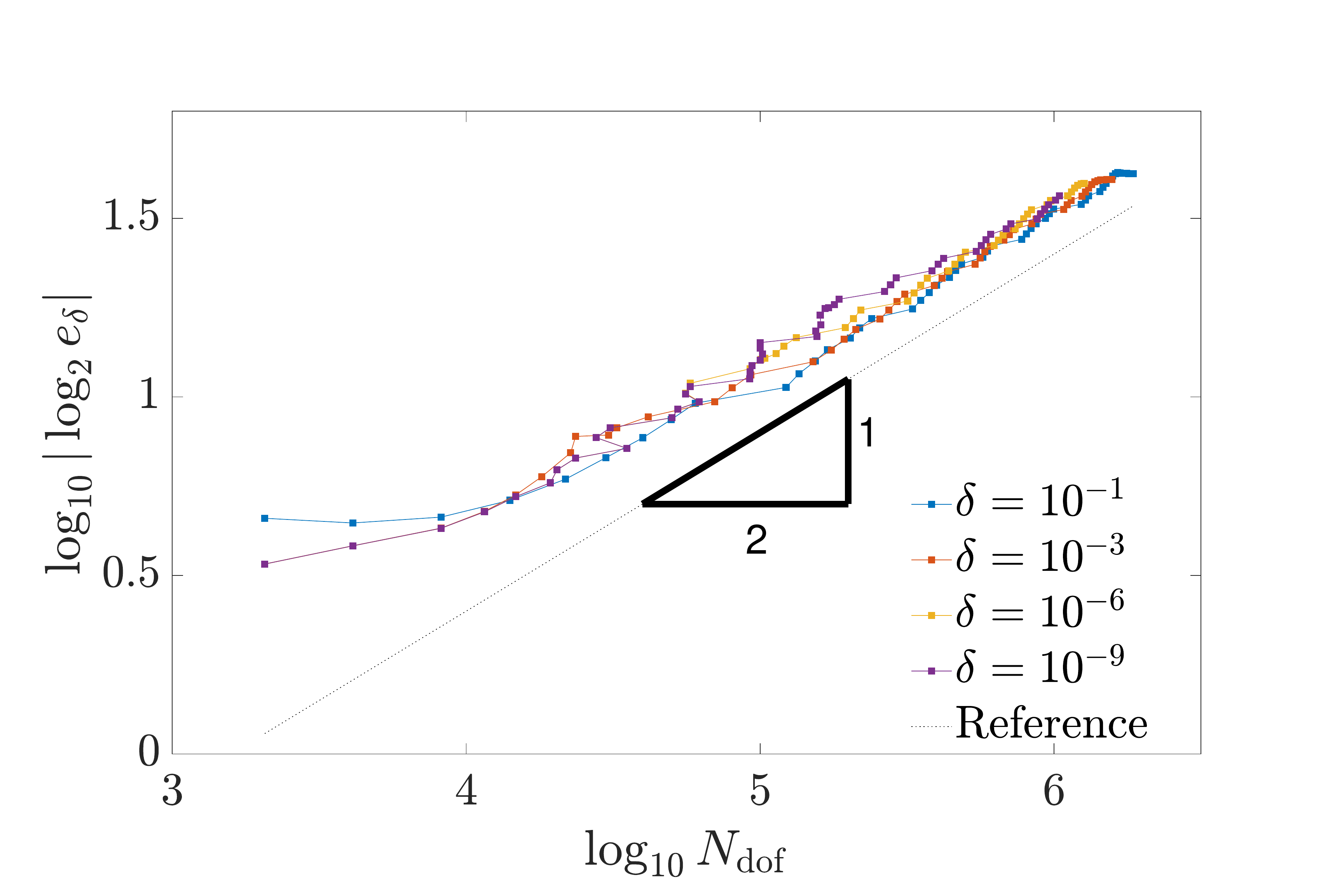}
\caption{Convergence rate coefficient estimation for $\delta
\in \{ 10^{-1}, 10^{-3}, 10^{-6}, 10^{-9} \}$.}
\label{fig:W_W_E_norm_dof_3_new}
\end{figure}
The same quantity is displayed in Figure \ref{fig:W_W_E_norm_dof_2_new}, as a
function of two different roots of the number of parameters of the QTT representation of the solution.
It appears that the exponent $1/3$ in equation \eqref{dofEstimate} of Theorem
\ref{QTTapprox} constitutes, in this case, an underestimation of the rate of convergence.
To properly study the experimental rate of convergence, we start from the ansatz
that there exist $C, b, \kappa>0$  such that
\begin{equation}\label{fitEq}
  e_{\delta}  = C 2^{-b
\left(\Ndof \right)^{\kappa}},
\end{equation}
then, assuming that $\left|  \log_2C\right|$ is small in comparison to $b \Ndof^\kappa$,
\begin{equation*}
  \left|\log_{2} e_{\delta}   \right|= \left|\log_{2} C - b \left(\Ndof
\right)^{\kappa}   \right|\approx  b \left(\Ndof \right)^{\kappa},
\end{equation*}
Hence
\begin{equation*}
  \log_{10}| \log_{2} e_{\delta} | \approx \log_{10} b + \kappa \log_{10} \Ndof,
\end{equation*}
We plot then, in Figure
\ref{fig:W_W_E_norm_dof_3_new}, $\log_{10}\left| \log_2\ed \right|$ as a
function of $\log_{10}\Ndof$, and estimate the exponent $\kappa$. In accordance with our
previous observation, the values obtained strongly indicate that $\kappa=1/2$
for this set of computations.

\section{Conclusions and future work}\label{sec:conclusion}
We have proved that the solutions of singularly perturbed PDEs in one dimension
admit low rank QTT decompositions and that the energy norm of the error
converges exponentially with respect to the third root of the number of
parameters of the QTT decomposition, independently of the perturbation parameter. This is confirmed in a
numerical test case, where we observe slightly better rates of convergence than
prescribed by the theory, with errors independent of the perturbation parameter
if an adaptive truncation strategy is chosen. Furthermore, the preconditioned
system is stable at all perturbation scales and allows for the resolution of the
boundary layer for very small perturbation parameters.

The natural extension of this work is the analysis of singularly perturbed
problem in higher physical dimensions. This will be the subject of future
investigation; we remark that our strategy to obtain rank bounds through high
order approximation and $L^2$ projection extends rather naturally to this case.

\appendix
\section{Assembly of preconditioner} \label{sec:assembly_prec}
We discuss here the details on the construction of the preconditioner.
We follow the construction in \cite{BachmayrStability}, while introducing a
modification to obtain stability independent of the perturbation parameter.

For all $0<\delta <1$, and for $f, c\in L^2(I)$, 
we introduce the singularly perturbed problem with homogeneous Dirichlet-Neumann boundary conditions of finding $\vd\in
H^1(I)$ such that 
\begin{equation}\label{eq:DN-sp}
  \begin{aligned}
    - \delta^{2} v_{\delta}'' + c v_{\delta} = f, &\text{ in } I,
    \\
    v_{\delta}(0)  = v_{\delta}'(1)=0.&
\end{aligned}
\end{equation}
Let then $\tA_L\in \mathbb{R}^{2^L\times 2^L}$ and $\tf_L\in \mathbb{R}^{2^L}$
be, respectively, the matrix and the right hand side corresponding to the
$\mathbb{P}_1$ FE discretization of \eqref{eq:DN-sp} on the grid with nodes
$\{j2^{-L}: j\in\range{2^L}\}$.

We introduce the following notation for the preconditioned system:
\begin{equation*}
  B_{L} \coloneqq C_{L} \tA_{L} C_{L},\qquad g_{L} \coloneqq C_{L}
\tf_{L},
\end{equation*}
where the matrix $C_L$ has been introduced in \eqref{eq:prec_perturbed}.
In order to solve the linear system $B_{L} \bar{u}_{L} = g_{L}$
in QTT format, we need to assemble $B_{L}$ and $g_{L}$. As
shown in Section 2.3 and 2.4 in \cite{BachmayrStability}, there exist
matrices $Q_{L,0}, Q_{L,1}, \Lambda_{L,0}$, and $\Lambda_{L,1}$ of small fixed QTT-rank such that
\begin{equation}\label{sysMPre}
  B_{L} = Q_{L,0}^{T} \Lambda_{L,0} Q_{L,0} +
Q_{L,1}^{T} \Lambda_{L,1} Q_{L,1}.
\end{equation}

In what follows, we use normal parentheses $(\ )$ to indicate
matrices and vectors and we use square brackets $[\ ]$ to indicate block
structures with matrices or vectors as elements. Superscript $T$ on a matrix
denotes the usual matrix transposition and superscript $T$ on a block structure
(with elements being matrices or vectors) refers to transposition of the
individual block elements.

\subsection{Operations on TT-cores}
We introduce some additional notation in order to present the explicit
construction of the preconditioner.
We represent TT-cores as block matrices and introduce two operations on TT-cores
denoted by $\bullet$ and $\Join$. In Definition \ref{TT-dec} we referred to
$V^{1}, \ldots ,V^{d} \in \mathbb{R}^{r \times n \times r}$ as the TT-cores of
$A$ with $r \in \mathbb{N}$ being the TT-rank of $A$ and $n \in
\mathbb{N}$ the mode size of $A$. We generalize the definition mentioned above
and introduce TT-cores of
ranks $p \times q$ and mode sizes $m \times n$, for $p, q, m, n\in \mathbb{N}$.
\begin{definition}[{\cite[Section 3.2]{BachmayrStability}}] Let $m,n,p,q \in
\mathbb{N}$ and let $U^{[\alpha,\beta]} \in \mathbb{R}^{m \times n}$ be tensors
of sizes $m \times n$, for $\alpha = 1, \ldots ,p$ and for $\beta = 1, \ldots
,q$. We call the $4$-tensor $U \in \mathbb{R}^{p \times m \times n \times q}$
defined by
\begin{equation*}
  U(\alpha, i,j, \beta) = U^{[\alpha,\beta]}_{ij},
\end{equation*}
for all $\alpha = 1, \ldots ,p$, $i=1, \ldots ,m$, $j = 1,
\ldots ,n$ and $\beta = 1, \ldots ,q$, a TT-core of ranks $p \times q$ and of
mode sizes $m \times n$. Given a $TT$-core $U$ of ranks $p \times q$ and of mode
sizes $m \times n$, we call each tensor $U^{[\alpha,\beta]}$, $\alpha=1, \ldots
,p$ and $\beta = 1, \ldots ,q$, the $(\alpha,\beta)$-block of $U$.
\end{definition}
As a TT-core $U$ of ranks $p \times q$ and of mode sizes $m
\times n$ is specified by the $2$-tensors $U^{[\alpha,\beta]}$, $\alpha = 1,
\ldots ,p$ and $\beta = 1, \ldots ,q$, for ease of exposition and for
convenience we can specify $U$ in the block structure 
\begin{equation}\label{coreBlock}
  U = \left[ \begin{matrix} U^{[1,1]} & \cdots &
U^{[1,q]} \\ \vdots & \ddots & \vdots \\ U^{[p,1]} & \cdots &
U^{[p,q]} \end{matrix}
\right].
\end{equation}
We use the representation in equation \eqref{coreBlock} whenever
we specify TT-cores. Then, by our convention for transposition, we have that
\begin{equation*}
  U^{T}(\alpha,i,j,\beta) = U(\alpha,j,i,\beta),
\end{equation*}
or equivalently,
\begin{equation*}
  (U^{T})^{[\alpha,\beta]} = (U^{[\alpha,\beta]})^{T}.
\end{equation*}
\begin{definition}[{\cite[Section 3.2]{BachmayrStability}}] Let $m,n,p,q \in
\mathbb{N}$ and let $U$ be a TT-core of ranks $p \times q$ and of mode sizes $m
\times n$. We define the $(i,j)$-\myemph{slice} $U^{\{ i,j \}} \in \mathbb{R}^{p
\times q}$ of $U$ as the matrix defined by
\begin{equation*}
  U^{\{ i,j \}}_{\alpha,\beta} = U(\alpha,i,j,\beta).
\end{equation*}
\end{definition}
In order to combine different TT-cores, we introduce two
operations $\bullet$ and $\Join$ on TT-cores.
\begin{definition}[{\cite[Definition 1]{BachmayrStability}}] Let $p,q,r \in
\mathbb{N}$ and let $m_{1},m_{2},n_{1},n_{2} \in \mathbb{N}$. Consider two
TT-cores $U$ and $V$ of ranks $p \times r$ and $r \times q$ and of mode $m_{1}
\times m_{2}$ and $n_{1} \times n_{2}$, respectively. The \myemph{strong Kronecker
product} $U \Join V$ of $U$ and $V$ is the TT-core of rank $p \times q$ and mode
size $m_{1} m_{2} \times n_{1} n_{2}$ given, in terms of matrix multiplication
of slices of sizes $p \times r$ and $r \times q$, by
\begin{equation*}
  (U \Join V)^{\{ i_{1} i_{2}, j_{1} j_{2} \}} \coloneqq
U^{\{i_{1},j_{1}\}} V^{\{i_{2},j_{2}\}}
\end{equation*}
for all combinations $i_{k} \in \{ m_{1}, \ldots ,m_{k}\}$ and
$j_{k} \in \{ 1, \ldots ,n_{k} \}$ with $k=1,2$.
\end{definition}
\begin{definition}[{\cite[Definition 2]{BachmayrStability}}] Let $p,p',r,r' \in
\mathbb{N}$ and let $m,n,k \in \mathbb{N}$. Consider two TT-cores $A$ and $B$ of
ranks $p \times p'$ and $r \times r'$ and of mode size $m \times k$ and $k
\times n$, respectively. The \myemph{mode core product} $A \bullet B$ of $A$ and
$B$ is the TT-core of rank $pr \times p'r'$ and mode size $m \times n$ given, in
terms of matrix multiplication of blocks of size $m \times k$ and $k \times n$,
by
\begin{equation*}
  (A \bullet B)^{\alpha \beta, \alpha' \beta'} \coloneqq
A^{[\alpha,\alpha']} B^{[\beta,\beta']}
\end{equation*}
for all combinations of $\alpha = 1, \ldots ,p$, $\alpha' = 1,
\ldots ,p'$, $\beta = 1, \ldots ,r$ and $\beta' = 1, \ldots ,r'$.
\end{definition}

\subsection{Construction of the preconditioner}
We start by introducing some building block that will be necessary for  
the explicit
TT-decompositions of the constituents of equation \eqref{sysMPre}:
\begin{equation*}
  A_{b} \coloneqq A \bullet A,\ U_{b} \coloneqq U \bullet
U^{T},\ X_{b} \coloneqq X \bullet X^{T},\ P_{b} \coloneqq P \bullet P,
\end{equation*}
where
\begin{equation*}
  U \coloneqq \left[ \begin{matrix} I & J^{T} \\ &
      J \end{matrix}
  \right],\ X \coloneqq \frac{1}{2} \left[ \begin{matrix}
\Bigg(\begin{matrix} 1 \\ 2 \end{matrix}
\Bigg) & \Bigg(\begin{matrix} 0 \\
1 \end{matrix}
\Bigg) \\ \Bigg(\begin{matrix} 1 \\ 0 \end{matrix}
\Bigg) &
\Bigg(\begin{matrix} 2 \\ 1 \end{matrix}
\Bigg) \end{matrix}
 \right],\ P
\coloneqq \left[\begin{matrix} 1 \\ 0 \end{matrix}
 \right]
\end{equation*}
\begin{equation*}
  I \coloneqq \left( \begin{matrix} 1 & 0 \\ 0 & 1 \end{matrix}
\right),\ J = \left( \begin{matrix} 0 & 1 \\ 0 & 0 \end{matrix}
 \right)
\end{equation*}
 and for $\alpha \in \{ 0,1 \}$
 \begin{equation*}
   W_{\alpha} \coloneqq T_{\alpha} \bullet \bar{I},\ Z_{\alpha}
\coloneqq Y_{\alpha} \bullet X^{T},\ K_{\alpha} \coloneqq N_{\alpha} \bullet P,
\end{equation*}
 where
 \begin{equation*}
   T_{1} \coloneqq \left[ \begin{matrix} 1 \\ 1 \end{matrix}
\right],\ \bar{I} \coloneqq \left[ \begin{matrix} 1 & 0 \\ 0 & 1 \end{matrix}
\right],\ Y_{0} \coloneqq \frac{1}{2} \left[ \begin{matrix} \Bigg(\begin{matrix}
1 \\ 1 \end{matrix}
 \Bigg) & \Bigg( \begin{matrix} 0 \\ 0 \end{matrix}
 \Bigg) \\
\Bigg(\begin{matrix} -1 \\ 1 \end{matrix}
\Bigg) & \Bigg(\begin{matrix} 1 \\
1 \end{matrix}
\Bigg) \end{matrix}
 \right]
\end{equation*}
\begin{equation*}
  Y_{1} \coloneqq \frac{1}{2} \left[ \Bigg(\begin{matrix} 1 \\
1 \end{matrix}
\Bigg) \right] ,\ N_{1} \coloneqq [1],\ N_{0} \coloneqq
\frac{1}{2} \left[ \begin{matrix} \Bigg(\begin{matrix}1 \\ 0 \end{matrix}
\Bigg)
\\ \Bigg(\begin{matrix} 0 \\ 1 \end{matrix}
\Bigg) \end{matrix}
 \right].
\end{equation*}
 Recall the different brackets $(\ )$ and $[\ ]$ and their
respective meanings related to construction of TT-cores. We now state the
TT-decompositions of the constituents of equation \eqref{sysMPre}.
\begin{lemma}\label{C-dec}
  For any $L \in \mathbb{N}$, the matrix $C_{L}$
admits the TT-decomposition
\begin{equation}\label{CL1}
  C_{L} = [\begin{matrix} A_{b} & \mu_{0,\delta}
A_{b} \end{matrix}
 ] \Join C_{1,L} \ldots \Join C_{L,L} \Join
\left[\begin{matrix} \\ P_{b} \end{matrix}
 \right],
\end{equation}
 with TT-ranks equal to $8$ and where
 \begin{equation}\label{CL2}
   C_{\ell,L} = \left[ \begin{matrix} U_{b} &
\bpxcoef U_{b} \\ & 2^{-1} X_{b} \end{matrix}
 \right] , \qquad \text{for all 
} \ell=1, \ldots ,L.
\end{equation}
\end{lemma}
\begin{proof} Recall the definition of $C_{L}$
  \begin{equation}\label{CL-def}
    C_{L} = \sum_{\ell=0}^{L} \bpxcoef
P_{\ell,L} P^{T}_{\ell,L}.
\end{equation}
 By equation (86) in \cite[Section 5.4]{BachmayrStability}, the
following representation holds for every $\ell = 0, \ldots ,L$
\begin{equation*}
  2^{-\ell} P_{\ell,L} P^{T}_{\ell,L} = 2^{-\ell} A_{b} \Join
U_{b}^{\Join \ell} \Join (2^{-1} X_{b})^{\Join (L-\ell)} \Join P_{b}.
\end{equation*}
 Thus, our modified expression is
 \begin{equation}
   \label{CL3}
   \bpxcoef P_{\ell,L} P^{T}_{\ell,L} = \bpxcoef A_{b} \Join U_{b}^{\Join \ell} \Join (2^{-1} X_{b})^{\Join (L-\ell)} \Join P_{b}
\end{equation}
 for $\ell = 0, \ldots ,L$. By inserting the expressions in
equation \eqref{CL2} into equation \eqref{CL1}, performing the multiplication
and inserting the expression in equation \eqref{CL3}, we end up with the
expression for $C_{L}$ as given in equation \eqref{CL-def}.
\end{proof}

\begin{lemma}\label{Q-dec}
  For any $L \in \mathbb{N}$ and for $\alpha \in \{
0,1 \}$, the matrix $Q_{L,\alpha}$ admits the TT-decomposition
\begin{equation}\label{QL1}
  Q_{L,\alpha} = [\begin{matrix} A_{b} & \mu_{0,\delta} A_{b}
\Join W_{\alpha} \end{matrix}
 ] \Join Q_{1} \Join \ldots \Join Q_{L} \Join
\left[\begin{matrix} \\ K_{\alpha} \end{matrix}
 \right]
\end{equation}
 with TT-ranks equal to $6$ and where
 \begin{equation}\label{QL2}
   Q_{\ell} = \left[ \begin{matrix} 2^{\frac{1}{2}}U_{b} &
\bpxcoef 2^{\alpha \ell+\frac{1}{2}} U_{b} \Join W_{\alpha} \\ &
2^{\alpha-\frac{1}{2}} Z_{\alpha} \end{matrix}
 \right], \text{ for } \ell = 1,
\ldots ,L.
\end{equation}
\end{lemma}
\begin{proof} By equation (33c) in \cite[Section 2.4]{BachmayrStability}, the
following relation hold for $\alpha \in \{ 0,1 \}$:
\begin{equation}\label{QL0}
  Q_{L,\alpha} = M_{L,\alpha} C_{L}.
\end{equation}
 Moreover, by equation (88) in \cite[Section
5.4]{BachmayrStability} the following representation holds for every $\ell = 0,
\ldots ,L$
\begin{equation*}
  2^{-\ell} M_{L,\alpha} P_{\ell,L} P^{T}_{\ell,L} =
2^{-(1-\alpha) \ell} A_{b} \Join (2^{\frac{1}{2}} U_{b})^{\Join \ell} \Join
W_{\alpha} \Join (2^{\alpha - \frac{1}{2}} Z_{\alpha})^{\Join (L-\ell)} \Join
K_{\alpha}
\end{equation*}
 Thus, our modified expression is
 \begin{multline}\label{QL3}
   \bpxcoef M_{L,\alpha} P_{\ell,L}
P^{T}_{\ell,L} = \bpxcoef 2^{\alpha \ell} A_{b} \Join
(2^{\frac{1}{2}} U_{b})^{\Join \ell} \Join W_{\alpha} \Join (2^{\alpha -
\frac{1}{2}} Z_{\alpha})^{\Join (L-\ell)} \Join K_{\alpha}
\end{multline}
 for $\alpha \in \{ 0,1 \}$ and for $\ell = 0, \ldots ,L$. By
inserting the expressions in equation \eqref{QL2} into equation \eqref{QL1}, we
end up with the same expression as when taking the sum over $\ell = 0, \ldots
,L$ of the expression in equation \eqref{QL3}. The result now follows from the
formula in equation \eqref{QL0} and equation \eqref{CL-def}:
\begin{equation*}
  Q_{L,\alpha} = M_{L,\alpha} C_{L} = \sum_{\ell=0}^{L}
\bpxcoef M_{L,\alpha} P_{\ell,L} P_{\ell,L}^{T}.
\end{equation*}
\end{proof}

The explicit TT-decomposition of $\Lambda_{L,1}$ in equation~\eqref{sysMPre} is presented below, while $\Lambda_{L,0}$ remains unchanged compared with~\cite{BachmayrStability}.
\begin{lemma}\label{L-dec}
  For any $L \in \mathbb{N}$ the matrix $\Lambda_{L,1}$ admits the TT-decomposition
\begin{equation}\label{LL1}
  \Lambda_{L,1} = \Lambda_{L,0} \Join
\Lambda_{L,1,1} \Join \ldots \Join \Lambda_{L,L,1} \Join
\Lambda_{L,L+1,1}
\end{equation}
where $\Lambda_{L,\ell,1}=\frac{\delta^{2/L}}{2} I$, for $\ell = 1, \ldots ,L$, are
TT-cores of ranks $1 \times 1$ and of mode sizes $2 \times 2$ and $\Lambda_{L,L+1,1}=1$ is a TT-core of ranks $1 \times 1$ and
of mode sizes $1 \times 1$.
\end{lemma}
\begin{proof} By straightforward computation of the quantities in equations
(90c) and (90c) in \cite[Section 5.4]{BachmayrStability} corresponding to the
system matrix given by $S_{L} + R_{L}$ (instead of $\delta^{2} S_{L} + R_{L}$),
we end up with equation \eqref{LL1} except that the intermediate TT-cores
$\Lambda_{L,\ell,1}$, for $\ell = 1, \ldots ,L$, are given by $\frac{1}{2} I$
(instead of $\frac{\delta^{2/L}}{2} I$). Lemma \ref{C-dec} and Lemma
\ref{Q-dec} do not include the factor of $\delta^{2}$ in the system matrix
$\delta^{2} S_{L} + R_{L}$. Clearly,
\begin{equation*}
  (\delta^{2/L})^{L} = \delta^{2},
\end{equation*}
and thus we obtain the decomposition for $\Lambda_{L,\alpha}$ of
the system matrix $\delta^{2} S_{L} + R_{L}$ by multiplying the obtained
expression for $\Lambda_{L,\ell,1}$, $\ell = 1, \ldots ,L$, with $\delta^{2/L}$.
This yields the desired decomposition as given in equation \eqref{LL1}.
\end{proof}
Lemmas \ref{C-dec}, \ref{Q-dec}, and \ref{L-dec} provide a complete explicit TT-representation of the
preconditioned system matrix $B_{L}$ in equation \eqref{sysMPre}.

\subsection{Application of preconditioner}
The preconditioner in the previous
subsection was introduced for Dirichlet-Neumann boundary value problems with
homogeneous boundary conditions. We show here how to apply it to the
  Dirichlet-Dirichlet case numerically approximated in Section \ref{sc:TT-numerics}. We suppose then, for
  ease of exposition, that
  the reaction coefficient is constant
  $c(x) \equiv \overline{c}\in \mathbb{R}$ for all $x\in I$. The first step is to instead consider the
following problem with homogeneous boundary conditions
\begin{equation}\label{diffEqMod}
  \begin{aligned}
    - \delta^{2} v_{\delta}'' + c v_{\delta} &= f - cg,\text{ in } (0,1)
\\ v_{\delta}(0) = 0,\ v_{\delta}(1)&=0,
\end{aligned}
\end{equation}
where $g(x) \coloneqq \alpha_{1} x - \alpha_{0} (x-1)$, whose
solution relates to the solution of \eqref{DiffEq} as $v_{\delta}(x) =
u_{\delta}(x)- g(x)$. 
We also introduce the corresponding Dirichlet-Neumann
problem
\begin{equation}\label{DNdiffEqMod}
\begin{aligned} - \delta^{2} v_{\delta}'' + c v_{\delta} &= f - c g,\text{ in }
(0,1) \\ v_{\delta}(0) = 0,\ v_{\delta}'(1)&=0,
\end{aligned}
\end{equation}
for which we have introduced the preconditioner $C_{L}$~\eqref{eq:prec_perturbed}. 

We use the well-known Sherman-Morrison formula \cite{Sherman-Morrison-general} to transfer the solution of the
preconditioned discretization of \eqref{DNdiffEqMod} to the case with
Dirichlet-Dirichlet boundary conditions as in \eqref{diffEqMod}.
\begin{theorem}[{\cite{Sherman-Morrison}}]\label{SM-theorem}
  Let $B \in
\mathbb{R}^{n \times n}$ be an invertible matrix and let $u,v \in
\mathbb{R}^{n}$ be such that also $B + u v^{T}$ is invertible. Then the inverse
of $B + u v^{T}$ is given by
\begin{equation}\label{SM-formula}
  (B + u v^{T})^{-1} = B^{-1} - \frac{B^{-1} u
v^{T} B^{-1}}{1 + v^{T} B^{-1} u}
\end{equation}
\end{theorem}

The following straightforward corollary suggests how to apply, in practice, the Sherman-Morrison formula for solving linear systems.
\begin{corollary}[{\cite[Corollary 2]{Sherman-Morrison}}]\label{SM-cor}
  Let $B
\in \mathbb{R}^{n \times n}$ and let $u,v,y \in \mathbb{R}^{n}$. Suppose that
$B$ and $B + uv^{T}$ are invertible matrices and suppose that $x_{1} \in
\mathbb{R}^{n}$ satisfies $B x_{1} = y$ and that $x_{2} \in \mathbb{R}^{n}$
satisfies $B x_{2} = u$. Then
\begin{equation}\label{SM-sol}
  x_{3} \coloneqq x_{1} - \frac{v^{T} x_{1}}{1 +
v^{T} x_{2}} x_{2}
\end{equation}
satisfies $(B + u v^{T}) x_{3} = y$.
\end{corollary}
We may express $A_L$, in terms of a rescaling of the parts of $\widetilde{A}_{L}$ (since the Dirichlet-Dirichlet and Dirichlet-Neumann cases have different mesh sizes) and of a rank-one correction term.
Therefore, the solution of the singularly perturbed
problem with homogeneous Dirichlet boundary conditions can be obtained using
Corollary \ref{SM-cor}.
In the case where $c$ is nonconstant, then the preconditioned
    solution for the Dirichlet-Dirichlet case can be obtained in a similar
    fashion, provided that the integrals arising in $\Lambda_{L,0}$ are
    computed using the grid $\TunifintL$.

\bibliographystyle{amsalpha-abbrv}
\bibliography{bibliography}

\end{document}